\documentclass[11pt, a4paper]{article}
\usepackage{latexsym}
\usepackage{amsmath}
\usepackage{amssymb}
\usepackage[all]{xy}
\usepackage{amsfonts}
\usepackage{verbatim}
\usepackage{amsthm}
\usepackage{mathrsfs}
\usepackage{epsfig}
\usepackage{xy}
\usepackage{array}
\usepackage{stmaryrd}
\usepackage{graphicx,color}
\usepackage{xcolor}
\usepackage{tikz}
\usetikzlibrary{arrows,calc}
\usepackage{etex}
\usepackage{mathdots}
\usepackage{float}
\usepackage{graphics}
\usepackage{pdflscape}
\usepackage{enumerate}
\usepackage{anysize,hyperref}
\input xypic
\xyoption{all}

\usepackage[perpage,symbol]{footmisc}
\usepackage{setspace}
\renewcommand{\cal}{\mathcal}
\def\A{\mathscr{A}}

\def\C{\mathscr{C}}
\def\s{\mathfrak{s}}

\def\op{^\mathrm{op}}
\def\Ab{\mathit{Ab}}
\def\del{\delta}
\def\dr{\ar@{->}[r]}
\def\S{\Sigma}
\def\ex{{\mathbf{ex}}}
\def\fx{{\mathbf{fx}}}
\def\T{\mathcal{T}}
\def\I{\mathcal {I}}
\def\X{\mathscr{X}}
\def\Y{\mathscr{Y}}

\def\x{\mathbf{x}}

\def\X{\mathscr{X}}
\def\Y{\mathscr{Y}}
\newcommand{\End}{\operatorname{End}\nolimits}

\def\add{\mathsf{add}\hspace{.01in}}
\def\Ext{\mbox{Ext}}
\def\Hom{\mbox{Hom}}
\def\pj{\mathsf{Proj\emph{-}Inj}}
\def\proj{\mathsf{Proj}}
\def\inj{\mathsf{Inj}}
\def\E{\mathbb{E}}

\def\Z{\mathbb{Z}}
\def\F{\mathcal{F}}
\newcommand{\Q}{\mathbb{Q}}

\begin{document}
\baselineskip=15pt
\begin{center}
{\Large{\bf Cluster subalgebras and cotorsion pairs in Frobenius extriangulated categories}

\bigskip
{\large Wen Chang
\footnote{Supported by the Fundamental Research Funds for the Central Universities (Grant No.\;1301030603) and the NSF of China (Grant No.\;11601295)},
Panyue Zhou\footnote{Supported by the Hunan Provincial Natural Science Foundation of China (Grant No.\;2018JJ3205)}
and Bin Zhu\footnote{Supported by the NSF of China (Grants No.11671221)}}}
\end{center}
\date{}

%\maketitle
\def\blue{\color{blue}}
\def\red{\color{red}}

\newtheorem{theorem}{Theorem}[section]
\newtheorem{lemma}[theorem]{Lemma}
\newtheorem{corollary}[theorem]{Corollary}
\newtheorem{proposition}[theorem]{Proposition}
\newtheorem{conjecture}{Conjecture}
\theoremstyle{definition}
\newtheorem{definition}[theorem]{Definition}
\newtheorem{question}[theorem]{Question}
\newtheorem{remark}[theorem]{Remark}
\newtheorem{remark*}[]{Remark}
\newtheorem{example}[theorem]{Example}
\newtheorem{example*}[]{Example}

\newtheorem{construction}[theorem]{Construction}
\newtheorem{construction*}[]{Construction}

\newtheorem{assumption}[theorem]{Assumption}
\newtheorem{assumption*}[]{Assumption}
\newtheorem{defn-prop}[theorem]{Definition-Proposition}
\baselineskip=17pt
\parindent=0.5cm

\begin{abstract}
\baselineskip=16pt
 Nakaoka and Palu introduced the notion of extriangulated categories by extracting the similarities between exact categories and triangulated categories. In this paper, we study cotorsion pairs in a Frobenius extriangulated category $\C$. Especially, for a $2$-Calabi-Yau extriangulated category $\C$ with a cluster structure, we describe the cluster substructure in the cotorsion pairs. For rooted cluster algebras arising from $\C$ with cluster tilting objects, we give a one-to-one correspondence between cotorsion pairs in $\C$ and certain pairs of their rooted cluster subalgebras which we call complete pairs. Finally, we explain this correspondence by an example relating to a Grassmannian cluster algebra.\\[0.5cm]
\textbf{Key words:} Frobenius extriangulated categories; 2-Calabi-Yau extriangulated (or triangulated) categories; Cotorsion pairs; Cluster algebras.\\[0.2cm]
\textbf{ 2010 Mathematics Subject Classification:} 16S99; 16S70; 18E30.
\medskip
\end{abstract}

\tableofcontents

\section*{Introduction}
Triangulated categories and exact categories are two fundamental structures in mathematics. They are also important tools in many {\red mathematical} branches.
It is well known that these two kinds of categories have some similarities, there are even direct connections between them. For example, by a classical result of Happel \cite{ha}, the stable category of a Frobenius category, which is a special exact category, is a triangulated category, where the triangulated structure inherits from the exact structure. By extracting the similarities between triangulated categories and exact categories, Nakaoka and Palu \cite{np} recently introduced the notion of extriangulated categories, whose extriangulated structures are given by $\E$-triangles with some axioms. Except triangulated categories and exact categories, there are many other examples for extriangulated categories \cite{np,zhz}.

In recent years, the categorification is a topic of general interest. Roughly speaking, given a mathematical structure, a categorification is to find a category {\red whose Grothendieck group or Hall algebra} has this structure. For example, some $2$-Calabi-Yau triangulated categories or exact stably $2$-Calabi-Yau categories \cite{bmrrt, gls2, gls3, birs, fk} have cluster structures, which categorify the cluster algebras with or without coefficients introduced by Fomin and Zelevinsky \cite{fz}. Recall that an exact category $\cal B$ is called \emph{exact stably $2$-Calabi-Yau} if it is Frobenius, that is, $\cal B$ has enough projectives and enough injectives, which coincide, and the stable category $\overline{\cal B}$, which is triangulated
[Ha], is $2$-Calabi-Yau. Examples of exact stably $2$-Calabi-Yau categories relating to categorification and cluster tilting theory, are the module category over the preprojective algebra of a Dynkin quiver, and the category of maximal Cohen-Macaulay modules over a
$3$-dimensional complete local commutative isolated Gorenstein singularity. In this paper we study extriangulated categories in the view of the categorification of cluster algebras.

The aim of this paper is to give a classification of cotorsion pairs in a $2$-Calabi-Yau extriangulated category, and study the cluster substructures in cotorsion pairs of an extriangulated category with a cluster structure, where a cotorsion pair consists of two subcategories which generate the whole extriangulated category by extensions given by $\E$-triangles. Here a $k$-linear {\red Hom-finite and Ext-finite }extriangulated category $(\C,\E,\s)$ is called $2$-Calabi-Yau if there is a functorially isomorphism $\E(x,y)\simeq D\E(y,x)$, for any $x,y\in \C$, where $D$ is {\red the} $k$-duality. We will give a one-to-one correspondence between the cotorsion pairs of a $2$-Calabi-Yau extriangulated category and the complete pairs of the rooted cluster subalgebras of the rooted cluster algebra categorified by the extriangulated category (see Section 5 for precise meaning). Here a rooted cluster algebra is just a cluster algebra associated with a fixed cluster. This correspondence is established recently in the case of triangulated category in \cite{cz}.

%This paper has four parts. We recall some preliminaries in the first part. In the second part, we firstly study the mutations of cluster tilting subcategories in $2$-Calabi-Yau Frobenius extriangulated categories, these subcategories form the cluster structure, and the mutations correspond to the mutations in cluster algebras. Then we give a classification of the cotorsion pairs in the $2$-Calabi-Yau extriangulated category. In the third part, We prove that the two subcategories in a cotorsion pair inherit the cluster structure in the extriangulated category, and then we construct the main correspondence between the cotorsion pairs and the complete pairs of the rooted cluster algebra. Finally, to explain the correspondence, we give an example of Grassmannians cluster algebra of type $E_6$ which is categorified by an exact stably $2$-Calabi-Yau category.

The main examples of our $2$-Calabi-Yau extriangulated categories are exact stably $2$-Calabi-Yau categories and $2$-Calabi-Yau triangulated categories. Applying the main results to
exact stably $2$-Calabi-Yau categories produces new results on classification of cotorsion pairs in an exact stably $2$-Calabi-Yau category, cluster substructure on cotorsion pairs and a new connection between cotorsion pairs and pairs of cluster subalgebras of a rooted cluster algebras which are categorified by these exact stably $2$-Calabi-Yau categories. In particular, one can get corresponding results on cluster subalgebras of cluster algebras categorified by module categories over preprojective algebras studied in \cite{gls2,gls3}. Applications to $2$-Calabi-Yau triangulated categories reproduce the results obtained recently by the first and the third author in \cite{cz}. The basic approach in the paper is to use the mutation pairs studied in \cite{zhz} to establish a relation between the cotorsion pairs in extriangulated {\red categories} and the cotorsion pairs in some triangulated quotient categories.

The paper is organized as follows: In Section $1$, we review some definitions and facts that we need to use. Moreover, we introduce the notion of {\red a} $2$-Calabi-Yau extriangulated category, and give some examples. In Section 2, we find a one-to-one correspondence between the cotorsion {\red pairs} of {\red a} $2$-Calabi-Yau extriangulated category and certain quotient {\red categories}. In Section 3, we introduce the notion of {\red a} cluster tilting subcategory  in {\red an} extriangulated category and discuss some properties. We also
study the mutation of cluster tilting subcategory in a $2$-Calabi-Yau extriangulated category. In Section 4, we give a classification of cotorsion pairs in {\red a} $2$-Calabi-Yau extriangulated category with cluster tilting objects. Then we prove that the two subcategories in a cotorsion pair inherit the cluster structure in {\red an} extriangulated category. In Section 5, after we recall some definitions and properties on cluster algebras and rooted cluster algebras,  we construct the main correspondence between the cotorsion pairs and the complete pairs of {\red a} rooted cluster algebra. In Section 6,  to explain the correspondence, we give an example of {\red Grassmannian} cluster algebra of type $E_6$ which is categorified by an exact stably $2$-Calabi-Yau category.

{\red We end this section with some conventions.  All additive categories considered in this paper are assumed to be Krull-Schmidt, i.e. any object is isomorphic to a finite direct sum of objects whose
endomoprhism  rings are local.  If $\X$ is a subcategory of an additive category $\cal A$, then we always assume that $\X$ is a full subcategory which is closed under taking isomorphisms, direct sums and direct summands.
Let $\cal A$ be additive category and $M\in\cal A$, we denote by $\add M$ the full subcategory of $\cal A$ consisting of direct summands of direct sum of finitely many copies of $M$. We say an extriangulated category $(\C,\E,\s)$ is $k-$linear if $\Hom(A,B) $ and $\E(A,B)$ are $k$-linear spaces, for any $A,B\in\C$, where $k$ is a field. A k-linear extriangulated category $(\C,\E,\s)$ is Hom-finite (or Ext-finite), if $\Hom(A,B) $ ($\E(A,B),$ respectively) is a finite dimensional $k$-vector space, for any $A,B\in\C$. }

\section{Preliminaries}
\setcounter{equation}{0}
In this section, we collect some definitions and results that will be used in this paper.
\subsection{Functorially finite subcategories}

Recall that a subcategories $\X$ of an additive category $\C$ is said to be \emph{contravariantly finite} in $\C$ if for every object $M$ of $\C$, there exists some $X$ in $\X$ and a morphism $f\colon X\to M$ such that for every $X'$ in $\X$ the sequence
$$\Hom_{\C}(X', X)\xrightarrow{f\circ}\Hom_{\C}(X',M)\xrightarrow{~~} 0$$
is exact. In this case $f$ is called a right $\X$-approximation. Dually, we {\red define} covariantly finite subcategories in $\C$ and left $\X$-approximations. Furthermore, a subcategory of $\C$ is said to be \emph{functorially finite} in $\C$ if it is both contravariantly finite and covariantly finite in $\C$.
A morphism $f\colon A \rightarrow B$ in $\C$ is \emph{right
minimal} if any endomorphism $g\colon A\rightarrow A$ such that $fg=f$ is an automorphism;
and \emph{left minimal} if any endomorphism $h\colon B\rightarrow B$ such that $hf=f$ is an automorphism.  For more details, we refer to \cite{ar}.

\subsection{Extriangulated categories}
Let $\C$ be an additive category. Suppose that $\C$ is equipped with a biadditive functor $$\E\colon\C\op\times\C\to\Ab.$$ For any pair of objects $A,C\in\C$, an element $\delta\in\E(C,A)$ is called an {\it $\E$-extension}. Thus formally, an $\E$-extension is a triplet $(A,\delta,C)$.
Let $(A,\del,C)$ be an $\E$-extension. Since $\E$ is a bifunctor, for any $a\in\C(A,A')$ and $c\in\C(C',C)$, we have $\E$-extensions
$$ \E(C,a)(\del)\in\E(C,A')\ \ \text{and}\ \ \ \E(c,A)(\del)\in\E(C',A). $$
We abbreviately denote them by $a_\ast\del$ and $c^\ast\del$.
For any $A,C\in\C$, the zero element $0\in\E(C,A)$ is called the spilt $\E$-extension.

\begin{definition}{\cite[Definition 2.3]{np}}\label{a1}
Let $(A,\del,C),(A',\del',C')$ be any pair of $\E$-extensions. A {\it morphism} $$(a,c)\colon(A,\del,C)\to(A',\del',C')$$ of $\E$-extensions is a pair of morphisms $a\in\C(A,A')$ and $c\in\C(C,C')$ in $\C$, satisfying the equality
$$ a_\ast\del=c^\ast\del'. $$
Simply we denote it as $(a,c)\colon\del\to\del'$.
\end{definition}

Let $A,C\in\C$ be any pair of objects. Sequences of morphisms in $\C$
$$\xymatrix@C=0.7cm{A\ar[r]^{x} & B \ar[r]^{y} & C}\ \ \text{and}\ \ \ \xymatrix@C=0.7cm{A\ar[r]^{x'} & B' \ar[r]^{y'} & C}$$
are said to be {\it equivalent} if there exists an isomorphism $b\in\C(B,B')$ which makes the following diagram commutative.
$$\xymatrix{
A \ar[r]^x \ar@{=}[d] & B\ar[r]^y \ar[d]_{\simeq}^{b} & C\ar@{=}[d]&\\
A\ar[r]^{x'} & B' \ar[r]^{y'} & C &}$$

We denote the equivalence class of $\xymatrix@C=0.7cm{A\ar[r]^{x} & B \ar[r]^{y} & C}$ by $[\xymatrix@C=0.7cm{A\ar[r]^{x} & B \ar[r]^{y} & C}]$. For any $A,C\in\C$, we denote as
$ 0=[A\xrightarrow{\binom{1}{0}}A\oplus C\xrightarrow{(0,\ 1)}C].$

\begin{definition}{\cite[Definition 2.9]{np}}\label{a2}
Let $\s$ be a correspondence which associates an equivalence class $\s(\del)=[\xymatrix@C=0.7cm{A\ar[r]^{x} & B \ar[r]^{y} & C}]$ to any $\E$-extension $\del\in\E(C,A)$. This $\s$ is called a {\it realization} of $\E$, if it satisfies the following condition:
\begin{enumerate}[$\bullet$]
\item Let $\del\in\E(C,A)$ and $\del'\in\E(C',A')$ be any pair of $\E$-extensions, with $$\s(\del)=[\xymatrix@C=0.7cm{A\ar[r]^{x} & B \ar[r]^{y} & C}],\ \ \ \s(\del')=[\xymatrix@C=0.7cm{A'\ar[r]^{x'} & B'\ar[r]^{y'} & C'}].$$
Then, for any morphism $(a,c)\colon\del\to\del'$, there exists $b\in\C(B,B')$ which makes the following diagram commutative.
\begin{equation}\label{t10}
\begin{array}{l}
$$\xymatrix{
A \ar[r]^x \ar[d]^a & B\ar[r]^y \ar[d]^{b} & C\ar[d]^c&\\
A'\ar[r]^{x'} & B' \ar[r]^{y'} & C' &}$$
\end{array}
\end{equation}
\end{enumerate}
In this case, we say that sequence $\xymatrix@C=0.7cm{A\ar[r]^{x} & B \ar[r]^{y} & C}$ {\it realizes} $\del$, whenever it satisfies $$\s(\del)=[\xymatrix@C=0.7cm{A\ar[r]^{x} & B \ar[r]^{y} & C}].$$
Remark that this condition does not depend on the choices of the representatives of the equivalence classes. In the above situation, we say that (\ref{t10}) (or the triplet $(a,b,c)$) {\it realizes} $(a,c)$.
\end{definition}

Now we recall the definition of extriangulated categories introduced recently by Nakaoka and Palu, which is {\red the} main object to study in the paper.

\begin{definition}{\cite[Definition 2.12]{np}}\label{a3}
We call the pair $(\E,\s)$ an {\it external triangulation} of $\C$ if it satisfies the following conditions:
\begin{itemize}
\item[{\rm (ET1)}] $\E\colon\C\op\times\C\to\Ab$ is a biadditive functor.
\item[{\rm (ET2)}] $\s$ is an additive realization of $\E$.
\item[{\rm (ET3)}] Let $\del\in\E(C,A)$ and $\del'\in\E(C',A')$ be any pair of $\E$-extensions, realized as
$$ \s(\del)=[\xymatrix@C=0.7cm{A\ar[r]^{x} & B \ar[r]^{y} & C}],\ \ \s(\del')=[\xymatrix@C=0.7cm{A'\ar[r]^{x'} & B' \ar[r]^{y'} & C'}]. $$
For any commutative square
$$\xymatrix{
A \ar[r]^x \ar[d]^a & B\ar[r]^y \ar[d]^{b} & C&\\
A'\ar[r]^{x'} & B' \ar[r]^{y'} & C' &}$$
in $\C$, there exists a morphism $(a,c)\colon\del\to\del'$ which is realized by $(a,b,c)$.
\item[{\rm (ET3)$\op$}] Let $\del\in\E(C,A)$ and $\del'\in\E(C',A')$ be any pair of $\E$-extensions, realized by
$$\xymatrix@C=0.7cm{A\ar[r]^{x} & B \ar[r]^{y} & C}\ \ \text{and}\ \ \ \xymatrix@C=0.7cm{A'\ar[r]^{x'} & B' \ar[r]^{y'} & C'}$$
respectively.
For any commutative square
$$\xymatrix{
A \ar[r]^x& B\ar[r]^y \ar[d]^{b} & C\ar[d]^c&\\
A'\ar[r]^{x'} & B' \ar[r]^{y'} & C' &}$$
in $\C$, there exists a morphism $(a,c)\colon\del\to\del'$ which is realized by $(a,b,c)$.

\item[{\rm (ET4)}] Let $(A,\del,D)$ and $(B,\del',F)$ be $\E$-extensions realized by
$$\xymatrix@C=0.7cm{A\ar[r]^{f} & B \ar[r]^{f'} & D}\ \ \text{and}\ \ \ \xymatrix@C=0.7cm{B\ar[r]^{g} & C \ar[r]^{g'} & F}$$
respectively. Then there exist an object $E\in\C$, a commutative diagram
$$\xymatrix{A\ar[r]^{f}\ar@{=}[d]&B\ar[r]^{f'}\ar[d]^{g}&D\ar[d]^{d}\\
A\ar[r]^{h}&C\ar[d]^{g'}\ar[r]^{h'}&E\ar[d]^{e}\\
&F\ar@{=}[r]&F}$$
in $\C$, and an $\E$-extension $\del^{''}\in\E(E,A)$ realized by $\xymatrix@C=0.7cm{A\ar[r]^{h} & C \ar[r]^{h'} & E},$ which satisfy the following compatibilities.
\begin{itemize}
\item[{\rm (i)}] $\xymatrix@C=0.7cm{D\ar[r]^{d} & E \ar[r]^{e} & F}$  realizes $f'_{\ast}\del'$,
\item[{\rm (ii)}] $d^\ast\del''=\del$,

\item[{\rm (iii)}] $f_{\ast}\del''=e^{\ast}\del'$.
\end{itemize}

\item[{\rm (ET4)$\op$}]  Let $(D,\del,B)$ and $(F,\del',C)$ be $\E$-extensions realized by
$$\xymatrix@C=0.7cm{D\ar[r]^{f'} & A \ar[r]^{f} & B}\ \ \text{and}\ \ \ \xymatrix@C=0.7cm{F\ar[r]^{g'} & B \ar[r]^{g} & C}$$
respectively. Then there exist an object $E\in\C$, a commutative diagram
$$\xymatrix{D\ar[r]^{f'}\ar@{=}[d]&E\ar[r]^{f}\ar[d]^{h'}&F\ar[d]^{g'}\\
D\ar[r]^{f'}&A\ar[d]^{h}\ar[r]^{f}&B\ar[d]^{g}\\
&C\ar@{=}[r]&C}$$
in $\C$, and an $\E$-extension $\del^{''}\in\E(C,E)$ realized by $\xymatrix@C=0.7cm{E\ar[r]^{h'} & A \ar[r]^{h} & C},$ which satisfy the following compatibilities.
\begin{itemize}
\item[{\rm (i)}] $\xymatrix@C=0.7cm{D\ar[r]^{d} & E \ar[r]^{e} & F}$  realizes $g'^{\ast}\del$,
\item[{\rm (ii)}] $\del'=e_\ast\del''$,

\item[{\rm (iii)}] $d_\ast\del=g^{\ast}\del''$.
\end{itemize}
\end{itemize}
In this case, we call $\s$ an $\E$-{\it triangulation of }$\C$, and call the triplet $(\C,\E,\s)$ an {\it externally triangulated category}, or for short, {\it extriangulated category} $\C$.
\end{definition}

For an extriangulated category $\C$, we use the following notation:

\begin{itemize}
\item A sequence $\xymatrix@C=0.41cm{A\ar[r]^{x} & B \ar[r]^{y} & C}$ is called a {\it conflation} if it realizes some $\E$-extension $\del\in\E(C,A)$.
\item A morphism $f\in\C(A,B)$ is called an {\it inflation} if it admits some conflation $\xymatrix@C=0.7cm{A\ar[r]^{f} & B \ar[r]& C}.$
\item A morphism $f\in\C(A,B)$ is called a {\it deflation} if it admits some conflation $\xymatrix@C=0.7cm{K\ar[r]& A \ar[r]^f& B}.$
\item If a conflation $\xymatrix@C=0.6cm{A\ar[r]^{x} & B \ar[r]^{y} & C}$ realizes $\del\in\E(C,A)$, we call the pair $(\xymatrix@C=0.41cm{A\ar[r]^{x} & B \ar[r]^{y} & C},\del)$ an {\it $\E$-triangle}, and write it in the following way.
$$\xymatrix{A\ar[r]^{x} & B \ar[r]^{y} & C\ar@{-->}[r]^{\del}&}$$
\item Let $\xymatrix{A\ar[r]^{x}&B\ar[r]^{y}&C\ar@{-->}[r]^{\delta}&}$ and $\xymatrix{A'\ar[r]^{x'}&B'\ar[r]^{y'}&C'\ar@{-->}[r]^{\delta'}&}$ be any pair of $\E$-triangles. If a triplet $(a,b,c)$ realizes $(a,c)\colon\del\to\del'$ as in $(\ref{t10})$, then we write it as
    $$\xymatrix{
A \ar[r]^x \ar[d]_{a} & B\ar[r]^y \ar[d]^{b} & C\ar@{-->}[r]^{\delta} \ar[d]^{c}&\\
A'\ar[r]^{x'} & B' \ar[r]^{y'} & C' \ar@{-->}[r]^{\delta'}&}$$
and call $(a,b,c)$ a {\it morphism of $\E$-triangles}.
\end{itemize}

\begin{example}\label{a4}
(1) Exact category $\cal B$ can be viewed as an extriangulated category. For the definition and basic properties of an exact category, see \cite{bu}. In fact, a biadditive functor $\E:=\Ext^1_{\cal B}\colon
\cal B\op\times\cal B\to \Ab$.  Let $A,C\in\cal B$ be any pair of objects. Define $\Ext^1_{\cal B}(C,A)$ to be the collection of all equivalence classes of short exact sequences of the form $\xymatrix@C=0.7cm{A\ar[r]^{x} & B \ar[r]^{y} & C}$. We denote the equivalence class by $[\xymatrix@C=0.7cm{A\ar[r]^{x} & B \ar[r]^{y} & C}]$ as before.
For any $\del=[\xymatrix@C=0.7cm{A\ar[r]^{x} & B \ar[r]^{y} & C}]\in\Ext^1_{\cal B}(C,A)$,
define the realization $\s(\del)$ of $[\xymatrix@C=0.7cm{A\ar[r]^{x} & B \ar[r]^{y} & C}]$ to be $\del$ itself.  For more details, see \cite[Example 2.13]{np}.
\vspace{2mm}

(2) Let $\C$ be an triangulated category with shift functor $[1]$.
Put $\E:=\C(-,-[1])$. For any $\delta\in\E(C,A)=\C(C,A[1])$, take a triangle
$$\xymatrix@C=0.7cm{A\ar[r]^{x} & B \ar[r]^{y} & C\ar[r]^{\del}&A[1]}$$
and define as $\s(\del)=[\xymatrix@C=0.7cm{A\ar[r]^{x} & B \ar[r]^{y} & C}]$. Then $(\C,\E,\s)$
is an extriangulated category. It is easy to see that extension closed subcategories of triangulated categories are also
extriangulated categories. For more details, see \cite[Proposition 3.22]{np}.
\vspace{2mm}

(3) Let $\C$ be an extriangulated category, and $\cal J$ a subcategory of $\C$.
If $\cal J\subseteq\proj(\C)\cap\inj(\C)$, where $\proj(\C)$ is the full category of projective objects in $\C$ and $\inj(\C)$ is the full category of injective objects in $\C$, then $\C/\cal J$ is an extriangulated category.
This construction gives extriangulated categories which are not exact nor triangulated in general.
For more details, see \cite[Proposition 3.30]{np}.
\end{example}

We recall some concepts from \cite{np}. Let $\C$ be an extriangulated category.
\begin{itemize}
\item An object $P\in\C$ is called {\it projective} if
for any $\E$-triangle $\xymatrix{A\ar[r]^{x}&B\ar[r]^{y}&C\ar@{-->}[r]^{\delta}&}$ and any morphism $c\in\C(P,C)$, there exists $b\in\C(P,B)$ satisfying $yb=c$.
We denote the full subcategory of projective objects in $\C$ by {\red $\proj(\C)$}. Dually, the full subcategory of injective objects in $\C$ is denoted by {\red $\inj(\C)$}.

\item We say $\C$ {\it has enough projectives}, if
for any object $C\in\C$, there exists an $\E$-triangle
$$\xymatrix{A\ar[r]^{x}&P\ar[r]^{y}&C\ar@{-->}[r]^{\delta}&}$$
satisfying $P\in\cal P$.  We can define the notion of having enough injectives dually.

\item $\C$ is said to be {\it Frobenius} if $\C$ has enough projectives and enough injectives
and if moreover the projectives coincide with the injectives. In this
case one has the quotient category $\overline{\C}$ of $\C$ by injectives, which is a triangulated category by \cite{np}. We refer to this category as the \emph{stable category} of $\C$.
\end{itemize}

\begin{remark}\label{a5}
\begin{itemize}
\item[(1)] If $(\C, \E, \s)$ is an exact category, then enough projectives and enough injectives agree with the
usual definitions.
\item[(2)] If $(\C, \E, \s)$ is a triangulated category, then
$\proj(\C)$ and $\inj(\C)$ consist of zero objects. Moreover it is Frobenius as an extriangulated category.
\end{itemize}
\end{remark}

\begin{example}\label{cc2}{\cite[Corollary 4.12]{zhz}}
Let $\C$ be a triangulated category with Auslander-Reiten translation $\tau$, and
$\X$ a functorially finite subcategory of $\C$, which satisfies $\tau\X=\X$.
 For any $A,C\in\C$, define {\red $\E'(C,A)\subseteq\C(C,A[1])$} to be the
collection of all equivalence classes of triangles of the form
$\xymatrix@C=0.7cm{A\ar[r]^f&B\ar[r]^{g}&C\ar[r]^{\del\;}&A[1]}$, where $f$ is $\X$-monic.
$\s'(\delta)=[A\xrightarrow{~f~}B\xrightarrow{~g~}C]$, for any $\del\in\E'(C,A)$.  Then $(\C,\E',\s')$ is a Frobenius extriangulated category whose projective-injective objects are precisely $\X$.
\end{example}

\subsection{Mutations in extriangulated categories}
We recall the notion of mutation pairs in extriangulated categories from \cite[Definition 3.2]{zhz}.

\begin{definition}
Let $\X, \cal A$ and $\cal B$ be subcategories of an extriangulated category $\C$, and $\X\subseteq\cal A$ and $\X\subseteq\cal B$. The pair $(\cal A,\cal B)$ is called an $\X$-\emph{mutation pair }if it satisfies
\begin{itemize}
\item[$(1)$] For any $A\in\cal A$, there exists an  $\E$-triangle
$$\xymatrix{A\ar[r]^{\alpha}&X\ar[r]^{\beta}&B\ar@{-->}[r]^{\delta}&}$$
where $B\in\cal B$, $\alpha$ is a left $\X$-approximation of $A$ and $\beta$ is a right $\X$-approximation of $B$.

\item[$(2)$] For any $B\in\cal B$, there exists an $\E$-triangle
$$\xymatrix{A\ar[r]^{\alpha}&X\ar[r]^{\beta}&B\ar@{-->}[r]^{\delta}&}$$
where $A\in\cal A$, $\alpha$ is a left $\X$-approximation of $A$ and $\beta$ is a right $\X$-approximation of $B$.
\end{itemize}
Note that if $\C$ is a triangulated category, $\X$-mutation pair is just the same as Liu-Zhu's definition
\cite[Definition 2.6]{lz}.
In addition, if $\X$ is a rigid subcategory of $\C$, i.e., $\Ext_{\C}^1(\X,\X)=0$, then $\X$-mutation pair is just the same as Iyama-Yoshino's definition \cite[Definition 2.5]{iy}.
\end{definition}

Let $\X\subseteq\cal A$ be subcategories of a category $\C$. We denote by $\cal A/\X$
the category whose objects are objects of $\cal A$ and whose morphisms are elements of
$\Hom_{\cal A}(A,B)/\X(A,B)$ for any $A,B\in\cal A$. Such category is called the quotient category
of $\cal A$ by $\X$. For any morphism $f\colon A\to B$ in $\cal A$, we denote by $\overline{f}$ the image of $f$ under the natural quotient functor $\cal A\to \cal A/\X$.

\begin{theorem}{\emph{\cite[Theorem 3.15]{zhz}}}\label{z1}
Let $\C$ be an extriangulated category and let $\X\subseteq\cal A$ be {\red two additive subcategories}
of $\C$. We assume that the following two conditions concerning $\cal A$ and $\X$:
\begin{itemize}
\item[\emph{(1)}] $\cal A$ is extension closed,

\item[\emph{(2)}] $(\cal A,\cal A)$ forms an $\X$-mutation pair.
\end{itemize}
We put
$\cal M:=\cal A /\X.$ Then the category $\cal M$ has the structure of a triangulated category
with respect to the following shift functor and triangles:
\begin{itemize}
\item[\emph{(I)}] For $A\in\cal A$, we take an $\E$-triangle
$$\xymatrix{A\ar[r]^{a}&X\ar[r]^{b\quad}&A\langle1\rangle\ar@{-->}[r]&,}$$
where $a$ is a left $\X$-approximation and $b$ is a right $\X$-approximation.
Then $\langle1\rangle$ gives a well-defined auto-equivalence of $\cal M$,
which is the \textbf{shift functor} of $\cal M$.
\item[\emph{(II)}] For an $\E$-triangle $\xymatrix{A\ar[r]^{f}&B\ar[r]^{g}&C\ar@{-->}[r]&,}$
with $A,B,C\in\cal A$ and $f$ is $\X$-monic, there exists the following commutative diagram of $\E$-triangles:
$$
\xymatrix{
A\ar[r]^f\ar@{=}[d]&B\ar[r]^g\ar[d]&C\ar@{-->}[r]\ar[d]^h&\\
A\ar[r]^{a}&X\ar[r]^{b\quad}&A\langle1\rangle\ar@{-->}[r]&}$$
Then we have a complex $A\xrightarrow{~\overline{f}~}B\xrightarrow{~\overline{g}~}C
\xrightarrow{~\overline{h}~}A\langle1\rangle$. We define \textbf{triangles} in
$\cal M$ as the complexes which are isomorphic to a complex obtained in this way.
\end{itemize}
\end{theorem}

\begin{lemma}\label{z4}
Assume that $(\cal U,\cal V)$ is {\red a} $\Y$-mutation pair in $\C$
such that $\cal U\vee \cal V\subseteq\cal A$, where $\cal U\vee \cal V$ is the smallest subcategory of $\C$ containing $\cal U$
and $\cal V$, and $\X\subseteq\Y$. Then $(\overline{\cal U},\overline{\cal V})$ is {\red a} $\overline{\Y}$-mutation pair in $\cal M$.
\end{lemma}

\proof For any $U\in\overline{\cal U}$, there exists an {\red $\E$-triangle}
$$\xymatrix{U\ar[r]^{f}&Y\ar[r]^{g}&V\ar@{-->}[r]&,}$$
where $f$ is a left $\Y$-approximation, $g$ is a right $\Y$-approximation and $V\in\cal V$.
Thus we have {\red the} following commutative diagram:
$$\xymatrix{U\ar[r]^{f}\ar@{=}[d]&Y\ar[r]^{g}\ar[d]^{y}&V\ar@{-->}[r]\ar[d]^{h}&\\
U\ar[r]^{a}&X\ar[r]^{b\quad}&U\langle1\rangle\ar@{-->}[r]&.}$$
It follows that
$$\xymatrix{U\ar[r]^{\overline{f}}&Y\ar[r]^{\overline{g}}&V\ar[r]^{\overline{h}\quad}&U\langle1\rangle}$$
is a triangle in $\cal M$.
It is easy to see that $\overline{f}$ is a left $\overline{\Y}$-approximation of $U$, $\overline{g}$ is a right $\overline{\Y}$-approximation of $V$ and $V\in\overline{\cal V}$.

Dually, we can show that for any $V'\in\overline{\cal V}$, there exists a triangle
$$\xymatrix{U'\ar[r]^{\overline{f'}}&Y'\ar[r]^{\overline{g'}}&V'\ar[r]^{\overline{h'}\quad}&U'\langle1\rangle}$$
in $\cal M$, where $\overline{f'}$ is a left $\overline{\Y}$-approximation of $U'$, $\overline{g'}$ is a right $\overline{\Y}$-approximation of $V'$ and $U'\in\overline{\cal U}$.
This shows that $(\overline{\cal U},\overline{\cal V})$ is {\red a} $\overline{\Y}$-mutation pair in $\cal M$. \qed
\vspace{4mm}

\subsection{$2$-Calabi-Yau extriangulated categories}

Motivated by the {\red definitions} of $2$-Calabi-Yau triangulated categories and exact stably $2$-Calabi-Yau categories, we define $2$-Calabi-Yau extriangulated categories.
\begin{definition}\label{y1}
Let $\C$ be a $k$-linear {\red Hom-finite and Ext-finite} extriangulated category over a field $k$. $\C$ is called $2$-\emph{Calabi-Yau} if there exists a bifunctorial isomorphism
$$\E(A,B)\simeq \mathsf{D}\E(B,A),$$
for any $A,B\in\C$, where $\mathsf{D}=\Hom_{k}(-,k)$ is the usual $k$-duality.
\end{definition}

We give some examples for $2$-Calabi-Yau extriangulated {\red categories}.

\begin{example}\label{y2}
~~
\begin{itemize}
\item Any exact stably $2$-Calabi-Yau category is a $2$-Calabi-Yau extriangulated category.
\item Any $2$-Calabi-Yau triangulated category is a $2$-Calabi-Yau extriangulated category.
\item Let $\C_Q$ be the cluster category of the path algebra $kQ$, where $Q$ is the quiver
$1\xrightarrow{~~}2\xrightarrow{~~}3$. We have the following
AR-quiver for $\C_Q$, where $S_i$ and $P_i$ denote the simple and projective modules
associated with vertex $i$ respectively.
$$
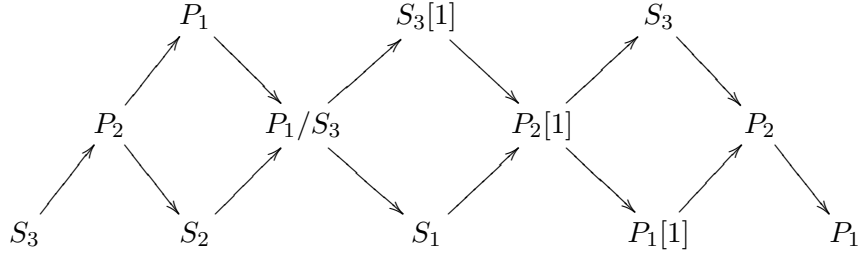
\xymatrix@C0.5cm{
&&P_1\ar[dr]&& S_3[1]\ar[dr]&&S_3\ar[dr]&&\\
&P_2\ar[ur]\ar[dr]&& P_1/S_3\ar[ur]\ar[dr]&& P_2[1]\ar[ur]\ar[dr]&& P_2\ar[dr]& \\
S_3\ar[ur]&& S_2\ar[ur]&& S_1 \ar[ur]&& P_1[1]\ar[ur] && P_1
}$$

Then $\cal B=\mathsf{mod}kQ$  is an extension closed subcategory of $\C_Q$. {\red
By \cite[Remark 2.18]{np}, $\cal B$ is an extriangulated category. It is easy to see $P_1$ is a projective (also injective) object, and $\cal B$ has enough projectives and enough injectives. Since $\C_Q$ is a $2$-Calabi-Yau triangulated category, }  then $\cal B$ is a $2$-Calabi-Yau extriangulated category.
\end{itemize}
\end{example}

\textbf{From now on to the end of the article, we always suppose that extriangulated category $\C$ has enough projectives
and enough injectives.}

\section{Cotorsion pairs in extriangulated categories }
\setcounter{equation}{0}

We recall the definition of {\red a} cotorsion pair in an {\red extriangulated} category from \cite{np}.

\begin{definition}\label{e1}{\cite[Definition 4.1]{np}}
Let $(\X,\Y)$ be subcategories of an extriangulated {\red category} $\C$. The pair $(\X,\Y)$ is called a {\it cotorsion pair} on $\C$ if it satisfies the following conditions.
\begin{enumerate}
\item[(1)] $\E(\X,\Y)=0$;
\item[(2)] For any $C\in\C$, there exists an $\E$-triangle
$$\xymatrix{Y^C\ar[r]^{f}&X^C\ar[r]^{g}&C\ar@{-->}[r]^{\del}&}$$
satisfying $X^C\in\X,Y^C\in\Y$.
\item[(3)] For any $C\in\C$, there exists an $\E$-triangle
$$\xymatrix{C\ar[r]^{f}&Y_C\ar[r]^{g}&X_C\ar@{-->}[r]^{\eta}&}$$
satisfying $X_C\in\X,Y_C\in\Y$.
\end{enumerate}
{\red Let $(\X,\Y)$ be a cotorsion pair,} we call $I(\X)=\X\cap\Y$ the \emph{core} of the cotorsion pair $(\X,\Y)$, which is usually denoted by $I$. Note that any projective-injective object belongs to the core $I$, thus the subcategory $\pj$ which consists of all projective-injective objects is contained in $I$. Moreover, if $I=\pj$, then we call $(\X,\Y)$ a \emph{t-structure}.
\end{definition}

Note that if $\C$ is a triangulated category, cotorsion pair is just the same as Iyama-Yoshino's definition \cite[Definition 2.2]{iy} or Nakaoka's definition \cite[Definition 2.1]{na}. If $\C$ is an exact category,  cotorsion pair is just the same as Liu's definition \cite[Definition 2.3]{l}.

\begin{remark}\label{rem-cotorsion}
Let $(\cal U,\cal V)$ be a cotorsion pair on an extriangulated category $\C$. Then
\begin{enumerate}
\item[(1)] $C\in\cal U$ if and only if $\E(C,\cal V)=0$;
\item[(2)] $C\in\cal V$ if and only if $\E(\cal U,C)=0$;
\item[(3)] $\cal U$ and $\cal V$ are extension closed;
\item[(4)] $\cal U$ is a contravariantly finite in $\C$ and $\cal V$ is a covariantly finite in $\C$.
\item[(5)] $\proj(\C)\subseteq\cal U$ and $\inj(\C)\subseteq\cal V$.
\end{enumerate}
\end{remark}

The following is an extriangulated version of Wakamatsu's Lemma.{\red
 This is proved in \cite[Lemma 2.1]{j1} in case $\C$ is a triangulated category, and
proved in \cite[Lemma 2.1]{j2}, when $\C$ is an $n$-angulated category. However, it can be easily extended to our setting.}

\begin{lemma}\label{d1}
Let $\X$ be an extension closed subcategory of an extriangulated category $\C$. Given an $\E$-triangle
$$\xymatrix{K\ar[r]^f&X_0\ar[r]^{g}&C\ar@{-->}[r]^{\del}&,}$$
where $g$ is a minimal right $\X$-approximation of $C$. Then $\E(X,K)=0$, for any $X\in\X$.
\end{lemma}

\proof There exists a long exact sequence
$$\C(X,X_0)\xrightarrow{\C(X,g)}\C(X,C)\xrightarrow{}\E(X,K)\xrightarrow{}\E(X,X_0)\xrightarrow{\E(X,g)}\E(X,C).$$
The first morphism in the sequence is epimorphism since $X_0\xrightarrow{g}C$ is a right $\X$-approximation, so the second morphism is zero.
We claim that the fourth map in the sequence is monomorphism whence the third map is zero. This forces $\E(X,K)=0$ as desired.

To see that the fourth map is {\red a} monomorphism, let $\eta\in\E(X,X_0)$ be any $\E$-extension, realized by
an $\E$-triangle $$\xymatrix{X_0\ar[r]^u&A\ar[r]^{v}&X\ar@{-->}[r]^{\eta}&}$$ such that
$\E(X,g)(\eta)=0$. Thus we obtain a morphism of $\E$-triangles
$$\xymatrix{X_0\ar[r]^{u}\ar[d]^g&A\ar[r]^v\ar[d]^h&X\ar@{-->}[r]^{\eta}\ar@{=}[d]&\\
C\ar[r]^{x}&B\ar[r]^{y}&X\ar@{-->}[r]^{0}&.}$$
By Corollary 3.5 in \cite{np}, there exists a morphism $x'\colon B\to C$ such that $x'x=1$.
Since $\X$ is extension closed, we have $A\in\X$. Since $g$ is a right $\X$-approximation, there exists a morphism $t\colon A\to X_0$ such that $gt=x'h$.
It follows that $g=x'xg=x'hu=gtu$. Since $g$ is right minimal, we have that $tu$ is an isomorphism.
In particular, $u$ is a section. By Corollary 3.5 in \cite{np}, $\eta$ splits and then $\eta=0$,  as desired.  \qed
\vspace{3mm}

We introduce two notations. Let $\X$ be a subcategory of an extriangulated category $\C$. We set
$\X^{\perp_{\E}}=\{M\in\C\mid \E(\X,M)=0\}$ and $^{\perp_{\E}}\X=\{M\in\C\mid \E(M,\X)=0\}$.
Remark that $\X^{\perp_{\E}}$ and $^{\perp_{\E}}\X$ are extension closed.

\begin{proposition}
Let $\X$ be a contravariantly finite extension closed subcategory of an extriangulated category $\C$ such that
$\proj(\C)\subseteq\X$. Then $(\X,\X^{\perp_{\E}})$ is a cotorsion pair in $\C$.
\end{proposition}

\proof For any object $C\in\C$, there exists an $\E$-triangle
$$\xymatrix{C\ar[r]^u&I_0\ar[r]^{v}&L\ar@{-->}[r]&,}$$
where $I_0\in\cal I$. Since $\X$ is contravariantly finite and $\proj(\C)\subseteq\X$, we can take two
$\E$-triangles
$$\xymatrix{M\ar[r]^x&X_1\ar[r]^{y}&L\ar@{-->}[r]&} \ \ \textrm{and} \ \ \xymatrix{K\ar[r]^f&X_2\ar[r]^{g}&C\ar@{-->}[r]&,}$$
where $y$ (resp.\,$g$) is a minimal right $\X$-approximation of $L$ (resp.\,$C$).
Since $\X$ is extension closed, by Lemma \ref{d1}, we obtain $M\in\X^{\perp_{\E}}$ (resp.\,$K\in\X^{\perp_{\E}}$).
By Proposition 3.15 in \cite{np}, we obtain a commutative diagram
$$\xymatrix{&M\ar@{=}[r]\ar[d]&M\ar[d]^{x}\\
C\ar[r]\ar@{=}[d]&N\ar[r]\ar[d]&X_1\ar@{-->}[r]\ar[d]^{y}&\\
C\ar[r]^{u}&I_0\ar[r]^{v}\ar@{-->}[d]&L\ar@{-->}[r]\ar@{-->}[d]&\\
&&}$$
of $\E$-triangles.
Since $I_0, M\in\X^{\perp_{\E}}$ and $\X^{\perp_{\E}}$ is extension closed, we have $N\in\X^{\perp_{\E}}$.
Thus for any object $C\in\C$, there exists two $\E$-triangles
$$\xymatrix{K\ar[r]&X_2\ar[r]&C\ar@{-->}[r]&}\ \ \textrm{and} \ \ \xymatrix{C\ar[r]&N\ar[r]&X_1\ar@{-->}[r]&}$$
satisfying $K,N\in\X^{\perp_{\E}}$ and $X_2, X_1\in\X$. Hence $(\X,\X^{\perp_{\E}})$ is a cotorsion pair.\qed
\vspace{1mm}

We omit the dual statement.

\begin{definition}
Let $\X$ be a subcategory of an extriangulated category $\C$.
$\X$ is called \emph{rigid} if $\E(\X,\X)=0$, i.e., $\E(A,B)=0$, for any $A,B\in\X$.
An object $X$ is called\emph{ rigid} if $\add X$ is rigid.

\end{definition}

\begin{lemma}\label{d2}
Let $\X$ be a functorially finite rigid subcategory of an extriangulated category $\C$ such that
$\X^{\perp_{\E}}={^{\perp_{\E}}\X}$ and $\proj(\C),\inj(\C)\subseteq\X$. Then   $(\X^{\perp_{\E}},\X^{\perp_{\E}})$ is an $\X$-mutation pair.
\end{lemma}

\proof
For any $M\in\X^{\perp_{\E}}$, there exists an $\E$-triangle
$$\xymatrix{K\ar[r]^f&X_0\ar[r]^{g}&M\ar@{-->}[r]^{\del}&,}$$
where $g$ is a minimal right $\X$-approximation of $M$.
By Lemma \ref{d1}, we have $\E(\X,K)=0$. Namely $K\in\X^{\perp_{\E}}={^{\perp_{\E}}\X}$ and
then $\E(K,\X)=0$.
Applying the functor $\C(-,\X)$ to the above $\E$-triangle, we obtain an exact sequence
$$\C(X_0,\X)\xrightarrow{\C(f,\X)}\C(K,\X)\xrightarrow{}\E(M,\X)=0.$$
This shows that $f$ is  a left $\X$-approximation of $K$.
Dually, we can show that $N\in\X^{\perp_{\E}}$, there exists an $\E$-triangle
$$\xymatrix{N\ar[r]^u&X_1\ar[r]^{v}&L\ar@{-->}[r]^{\del'}&,}$$
where $u$ is a left $\X$-approximation of $N$ and $v$ is a right $\X$-approximation of $L$.

Therefore, $(\X^{\perp_{\E}},\X^{\perp_{\E}})$ is an $\X$-mutation pair. \qed

\begin{lemma}\label{d3}
Let $\X$ be a subcategory of an extriangulated category $\C$ such that $\X^{\perp_{\E}}={^{\perp_{\E}}\X}$. If $(\cal U,\cal V)$ is a cotorsion pair on {\red $\C$} and $\X\subseteq\cal U\subseteq\X^{\perp_{\E}}$, then $\X\subseteq\cal V\subseteq\X^{\perp_{\E}}$.
\end{lemma}

\proof Since $\cal U\subseteq\X^{\perp_{\E}}={^{\perp_{\E}}\X}$, we have $\E(\cal U,\X)=0$ and then $\X\subseteq \cal V$.
Since $\X\subseteq\cal U$ and $\E(\cal U,\cal V)=0$, we have $\E(\X,\cal V)=0$ and then $V\subseteq\X^{\perp_{\E}}$.

\begin{lemma}\label{e5}
Let $\C$ be an extriangulated category, $\X\subseteq\cal A$ subcategories
of $\C$.  Assume that  $(\cal A,\cal A)$ forms {\red an} $\X$-mutation pair and $A,B$ are two objects in $\cal A$. If $\X$ is a rigid subcategory of $\C$, then $\E(A,B)\simeq\cal M(A,B\langle1\rangle)$.
\end{lemma}

\proof
Since $B\in\cal A$, there exists an $\E$-triangle in $\C$
$$\xymatrix{B\ar[r]^f&X_B\ar[r]^g&B\langle1\rangle\ar@{-->}[r]^{\delta}&,}$$
where $X_B\in\X$ and $g$ is a right $\X$-approximation of $B\langle1\rangle$. Applying the functor $\C(A,-)$ to this $\E$-triangle, we have the following exact sequence
$$\C(A,X_B)\xrightarrow{\C(A,g)}\C(A,B\langle1\rangle)\xrightarrow{}\E(A,B)\xrightarrow{}0$$
as $\E(A,X_B)=0$ by \cite[Lemma 3.5]{zhz}.

Note that $\textrm{Im}\,\C(A,g)=\X(A,B\langle1\rangle)$. It follows that
 $$\cal M(A,B\langle1\rangle)=\C(A,B\langle1\rangle)/\X(A,B\langle1\rangle)=\C(A,B\langle1\rangle)/\textrm{Im}\,\C(A,g)\simeq\E(A,B).$$   \qed
\vspace{2mm}

The following theorem gives a one-to-one correspondence between cotorsion pairs whose core
containing $\X$ in $\C$ and cotorsion pairs in $\cal M$, which generalizes Theorem 3.5 in \cite{zz1}. In the following, $\overline{\cal U}$ denotes the subcategory of $\cal M$
consisting of objects $U\in\cal U$.

\begin{theorem}\label{h2}
{\red Let $\C$ be an extriangulated category with  enough projectives and enough injectives},
and $\X$ be a  functorially finite rigid subcategory of $\C$ such that
$\cal A:=\X^{\perp_{\E}}={^{\perp_{\E}}\X}$ and $\cal P,\cal I\subseteq\X$. If $\X\subseteq\cal U\subseteq\X^{\perp_{\E}}$
. Then $(\cal U ,\cal V)$ is a cotorsion pair  with core $I$ in $\C$ if and only if $(\overline{\cal U},\overline{\cal V})$  is a cotorsion pair  with the core $\overline{I}$ in $\cal M:=\cal A/\X$.
\end{theorem}

\proof By Lemma \ref{d2}, we know that $(\cal A,\cal A)$ is an $\X$-mutation pair. Since $\cal A$ is
extension closed, by Theorem \ref{z1}, we have that $\cal M$ is a triangulated category.

We first show the ``only if'' part.  Assume that $(\cal U ,\cal V)$ is a cotorsion pair in $\C$. Then $\E(\cal U,\cal V)=0$. By Lemma \ref{e5}, we have $\cal M(\cal U,\cal V\langle1\rangle)=0$.

For any $A\in\cal A$, there exists an $\E$-triangle in $\C$
$$\xymatrix{V\ar[r]^f&U\ar[r]^g&A\ar@{-->}[r]^{\del}&,}$$
where $V\in\cal V$ and $U\in\cal U$ as $(\cal U ,\cal V)$ is a cotorsion pair in $\C$. Applying the functor $\C(-,\X)$ to this $\E$-triangle, we have the following exact sequence
$$\C(U,\X)\xrightarrow{ {\red \C(f,\X)}}\C(V,\X)\xrightarrow{}\E(A,\X)=0.$$
It follows that $f$ is $\X$-monic.
Thus we get the following commutative diagram:
$$\xymatrix{V\ar[r]^{f}\ar@{=}[d]&U\ar[r]^g\ar[d]^b&A\ar@{-->}[r]^\del\ar[d]^h&\\
V\ar[r]^{\alpha}&X\ar[r]^{\beta}&V\langle1\rangle\ar@{-->}[r]^{\eta}&.}$$
Since $V,U,A$ are in $\cal A$, there exists
a standard triangle in $\cal M$:
$$\xymatrix{V\ar[r]^{\overline{f}}&U\ar[r]^{\overline{g}}&A\ar[r]^{\overline{h}\quad}&V\langle1\rangle.}$$
Therefore, $(\cal U,\cal V)$  is a cotorsion pair in $\cal M$ .
\vspace{2mm}

To prove the ``if'' part.  Assume that $(\cal U ,\cal V)$ is a cotorsion pair in $\cal M$. Then
$\cal M(\cal U,\cal V\langle1\rangle)=0$. By Lemma \ref{e5}, we have $\E(\cal U,\cal V)=0$.

For any $A\in\cal A$, there exist two standard triangles  in $\cal M$
$$\xymatrix{V_1\ar[r]^{\overline{f_1}}&U_1\ar[r]^{\overline{g_1}}&A\ar[r]^{\overline{h_1}\quad}&V_1\langle1\rangle}$$
$$\xymatrix{U_2\ar[r]^{\overline{f_2}\quad}&A\langle1\rangle\ar[r]^{\overline{g_2}}&V_2\langle1\rangle\ar[r]^{\overline{h_2}}&U_2\langle1\rangle,}$$
where $U_1,U_2\in\cal U$ and $V_1,V_2\in\cal V$ as $(\cal U ,\cal V)$ is a cotorsion pair in $\cal M$ and then
$$\xymatrix{V_1\ar[r]^{\overline{f_1}}&U_1\ar[r]^{\overline{g_1}}&A\ar[r]^{\overline{h_1}\quad}&V_1\langle1\rangle}$$
$$\xymatrix{A\ar[r]^{\quad\overline{ d_2}\quad}&V_2\ar[r]^{\overline{e_2}}&U_2\ar[r]^{\overline{f_2}}&A\langle1\rangle,}$$
are also two standard triangles in $\cal M$.
We may assume that it is induced by the following commutative diagrams of $\E$-triangles in $\C$.
$$\xymatrix{V_1\ar[r]^{f_1}\ar@{=}[d]&U_1\ar[r]^{g_1}\ar[d]^{c_1}&A\ar@{-->}[r]^{\del_1}\ar[d]^{h_1}&\\
V_1\ar[r]^{\alpha_1}&X_1\ar[r]^{\beta_1}&V_1\langle1\rangle\ar@{-->}[r]^{\del'_1}&,} \ \
\xymatrix{A\ar[r]^{d_2}\ar@{=}[d]&V_2\ar[r]^{e_2}\ar[d]^{c_2}&U_2\ar@{-->}[r]^{\del_2}\ar[d]^{f_2}&\\
A\ar[r]^{\alpha_2}&X_2\ar[r]^{\beta_2}&A\langle1\rangle\ar@{-->}[r]^{\del'_2}&.}$$
It follows that
$$\xymatrix{V_1\ar[r]^{f_1}&U_1\ar[r]^{g_1}&A\ar@{-->}[r]^{\del_1}&},\ \
 \xymatrix{A\ar[r]^{d_2}&V_2\ar[r]^{e_2}&U_2\ar@{-->}[r]^{\del_2}&,}$$
  are $\E$-triangles in $\cal A$, where $U_1,U_2\in\cal U$ and $V_1,V_2\in\cal V$.

For any $C\in\C$, there exists an $\E$-triangle
$$\xymatrix{K\ar[r]^f&X_0\ar[r]^{g}&C\ar@{-->}[r]^{\del}&,}$$
where $g$ is a minimal right $\X$-approximation of $C$.
By Lemma \ref{d1}, we have that $\E(\X,K)=0$ and then $K\in\cal A$.
Thus there exists an $\E$-triangle in $\C$.
$$\xymatrix{K\ar[r]^{d_2}&V_2\ar[r]^{e_2}&U_2\ar@{-->}[r]^{\del_2}&,}$$
where $V_2\in\cal V$ and $U_2\in\cal U$.
By Proposition 3.15 in \cite{np}, we obtain a commutative diagram
$$\xymatrix{
K\ar[r]^{f}\ar[d]^{d_2}&X_0\ar[r]^{f'}\ar[d]&C\ar@{-->}[r]^{\delta}\ar@{=}[d]&\\
V_2\ar[r]^{a}\ar[d]^{e_2}&U\ar[r]^b\ar[d]&C\ar[r]^{\eta}&\\
U_2\ar@{=}[r]\ar@{-->}[d]^{\del_g}&U_2\ar@{-->}[d]^{\C}&\\
&&}$$
of $\E$-triangles. Since $\cal U$ is extension closed, we have $U\in\cal U$.
Namely, for any $C\in\C$, there exists an $\E$-triangle
$$\xymatrix{V_2\ar[r]^{a}&U\ar[r]^{b}&C\ar@{-->}[r]^{\eta}&,}$$
where $V_2\in\cal V$ and $U\in\cal U$.

Similarly, we can show that there exists an $\E$-triangle
$$\xymatrix{C\ar[r]^{c}&V_3\ar[r]^{d}&U_3\ar@{-->}[r]^{\theta}&,}$$
where $V_3\in\cal V$ and $U_3\in\cal U$.
Therefore, {\red $(\cal U,\cal V)$}  is a cotorsion pair in $\C$.
\vspace{3mm}

Finally, we have that {\red $I(\overline{\cal U})=\overline{\cal U}\cap\overline{\cal V}=\overline{\cal U\cap\cal V}=\overline{I(\cal U)}$. } \qed
\vspace{2mm}

This theorem immediately yields the following.

\begin{corollary}\label{corollary:correspondence between cotorsion pairs}
Let $\C$ be a Frobenius extriangulated category.
Then $(\cal U ,\cal V)$ is a cotorsion pair with core $I$ in $\C$ if and only if $(\overline{\cal U},\overline{\cal V})$  is a cotorsion pair  with the core $\overline{I}$ in $\overline{\C}$.
\end{corollary}

\proof It is easy to see that $(\C,\C)$ forms an $\cal I$-mutation pair. This follows from Theorem \ref{h2} and Theorem \ref{z1}. \qed
\vspace{2mm}

 Apply to  exact stably $2$-Calabi-Yau categories (e.g. module categories over preprojective algebras of Dykin quivers or the subcategories $\mathcal C_M$ of module categories over preprojective algebras, where $M$ is a terminal module, for details see \cite{gls1,gls2}), we have the following correspondence between cotorsion pairs in an exact stably $2$-Calabi-Yau category and in its stable category.
\begin{corollary}
Let $\cal B$ be an exact stably $2$-Calabi-Yau category. Then $(\cal U ,\cal V)$ is a cotorsion pair with core $I$ in $\cal B$ if and only if $(\overline{\cal U},\overline{\cal V})$  is a cotorsion pair  with the core $\overline{I}$ in $\overline{\cal B}$.
\end{corollary}

\section{Mutations of cluster tilting subcategories}
\setcounter{equation}{0}

\begin{definition}\label{y5}
Let $\C$ be an extriangulated category, $\X$ a subcategory of $\C$.
\begin{itemize}
\item $\X$ is called \emph{cluster tilting} if it satisfies the following conditions:
\begin{enumerate}
\item[(1)] $\X$ is functorially finite in $\C$;
\item[(2)] $M\in \X$ if and only if $\E(M,\X)=0$;
\item[(3)] $M\in \X$ if and only if $\E(\X,M)=0$.
\end{enumerate}

\item  An object $X$ is called \emph{cluster tilting}  if $\add X$ is cluster tilting.
\end{itemize}
\end{definition}

This definition is a generalization of cluster tilting subcategories in triangulated categories \cite{bmrrt,kr,kz,iy,b} and in exact categories \cite{gls2,gls3,i}.

By definition of a cluster tilting subcategory, we can immediately conclude:
\begin{remark}\label{y6}
{\red Let $\C$ be an extriangulated category with  enough projectives and injectives.}
\begin{itemize}
\item  If $\X$ is a cluster tilting  subcategory of $\C$, then $\proj(\C)\subseteq\X$ and $\inj(\C)\subseteq\X$.
\item  $\X$ is a cluster tilting subcategory of $\C$ if and only if
\begin{enumerate}
\item[(1)] $\X$ is rigid;
\item[(2)] For any $C\in\C$, there exists an $\E$-triangle $\xymatrix@C=0.5cm{C\ar[r]^{a\;} & X_1 \ar[r]^{b} & X_2\ar@{-->}[r]^{\del}&}$, where $X_1,X_2\in\X$;
\item[(3)] For any $C\in\C$, there exists an $\E$-triangle $\xymatrix@C=0.5cm{X_3\ar[r]^{c\;} & X_4 \ar[r]^{d} & C\ar@{-->}[r]^{\eta}&}$, where $X_3,X_4\in\X$.
\end{enumerate}
\end{itemize}
\end{remark}

\begin{proposition}\label{z6}
Let $\X$ be a functorially finite rigid subcategory of a $2$-Calabi-Yau extriangulated category $\C$ such that $\proj(\C)\subseteq\X$.
Then $\cal N:=\X^{\perp_{\E}}/\X$ is a $2$-Calabi-Yau triangulated category.
\end{proposition}

\proof This follows from Lemma \ref{d2}, Theorem \ref{z1} and Lemma \ref{e5}.  \qed

\vspace{2mm}

The following result is easy to verify and will be used in the sequel.

\begin{lemma}\label{z7}
Let $\C$ be an additive category and $\X\subseteq\cal A$ two subcategories of $\C$.
\begin{enumerate}
\item[\emph{(1)}] If $\X$ is contravariantly finite in $\C$, then $\cal A$ is contravariantly finite in $\C$
if and only if $\cal A/\X$ is contravariantly finite in $\C/\X$.
\item[\emph{(2)}] If $\X$ is covariantly finite in $\C$, then $\cal A$ is covariantly finite in $\C$
if and only if $\cal A/\X$ is covariantly finite in $\C/\X$.
\item[\emph{(3)}] If $\X$ is functorially finite in $\C$, then $\cal A$ is functorially finite in $\C$
if and only if $\cal A/\X$ is functorially finite in $\C/\X$.
\end{enumerate}
\end{lemma}

\proof It is straightforward to check.  \qed

\begin{theorem}\label{z8}
Let $\C$ be a $2$-Calabi-Yau extriangulated category, and $\X$  a functorially finite rigid subcategory of $\C$ such that $\proj(\C)\subseteq\X$. Denote $\cal N:=\X^{\perp_{\E}}/\X$. The correspondence $\cal  R\longmapsto\overline{\cal R}:=\cal R/\X$ gives
\begin{enumerate}
\item[\emph{(1)}] a one-one correspondence between rigid subcategories of $\C$ containing $\X$ and rigid subcategories of $\cal N$, and
\item[\emph{(2)}] a one-one correspondence between cluster tilting subcategories of $\C$ containing $\X$ and cluster tilting subcategories of $\cal N$.
\end{enumerate}
\end{theorem}

\proof Obviously, any rigid subcategories $\cal R$ of $\C$ containing $\X$ is contained in $\X^{\perp_{\E}}={^{\perp_{\E}}\X}$.

(1). This follows from Lemma \ref{e5}.
\vspace{1.5mm}

(2). By Lemma \ref{e5}, it suffices to show that $\cal R$ is a functorially finite subcategory of $\C$ if and only if
$\overline{\cal R}$ is a  functorially finite subcategory of $\cal N$.
\vspace{2mm}

Since $\cal R$ and $\X$ are functorially finite subcategories of $\C$.
By Lemma \ref{z7}, we have that $\overline{\cal R}$ is a functorially finite subcategory of $\cal N:=\X^{\perp_{\E}}/\X$.

Conversely, for any $C\in\C$, since $\X$ is a  functorially finite subcategory of $\C$, there exists an $\E$-triangle
$$\xymatrix{K\ar[r]^{f}&X_0\ar[r]^{g}&C\ar@{-->}[r]^{\del}&,}$$
where $g$ is a right $\X$-approximation of $C$.   Applying the functor $\Hom_{\C}(\X,-)$ to the above $\E$-triangle, we have the following exact sequence
$$\Hom_{\C}(\X,X_0)\xrightarrow{\Hom_{\C}(\X,\ g)}\Hom_{\C}(\X,C)\xrightarrow{}\E(\X,K)\xrightarrow{}\E(\X,X_0)=0.$$
Since $g$ is a right $\X$-approximation of $C$, we have that $\Hom_{\C}(\X,g)$ is an epimorphism.
It follows that $\E(\X,K)=0$ and then $K\in\X^{\perp_{\E}}$.
Since $\overline{\cal R}$ is a cluster tilting subcategory of $\cal N$, there exists a triangle
$$\xymatrix{K\ar[r]^{u}&R_0\ar[r]^{v}&R_1\ar[r]^{w\quad}&K\langle1\rangle,}$$
in $\cal N$, where $R_0,R_1\in\cal R$. Without loss of generality, we can assume that the above triangle can be induced by
this $\E$-triangle
$$\xymatrix{K\ar[r]^{u}&R_0\ar[r]^{v}&R_1\ar@{-->}[r]&,}$$
where $R_0,R_1\in\cal R$. By Proposition 3.15 in \cite{np}, we obtain a commutative diagram
$$\xymatrix{
K\ar[r]^{f}\ar[d]^{u}&X_0\ar[r]^{g}\ar[d]&C\ar@{-->}[r]\ar@{=}[d]&\\
R_0\ar[r]^{a}\ar[d]^{v}&M\ar[r]^b\ar[d]&C\ar[r]&\\
R_1\ar@{=}[r]\ar@{-->}[d]&R_1\ar@{-->}[d]&\\
&&}$$
of $\E$-triangles in $\C$.
Since $\cal R$ is rigid and $X_0,R_1\in\cal R$, we have $M\in\cal R$.
Applying the functor $\Hom_{\C}(\cal R,-)$ to this $\E$-triangle
$$\xymatrix{R_0\ar[r]^{a}&M\ar[r]^{b}&C\ar@{-->}[r]&,}$$
we have the following exact sequence
$$\Hom_{\C}(\cal R,M)\xrightarrow{\Hom_{\C}(\cal R,\ b)}\Hom_{\C}(\cal R,C)\xrightarrow{~~}\E(\cal R,R_0)=0.$$
This shows that $\Hom_{\C}(\cal R, b)$ is an epimorphism.
Thus $\cal R$ is a contravariantly finite subcategory of $\C$.
Similarly, we can show that $\cal R$ is a covariantly finite subcategory of $\C$.
Therefore, $\cal R$ is a functorially finite subcategory of $\C$.
  \qed

\begin{theorem}\label{z9}
Let $\C$ be a $2$-Calabi-Yau extriangulated category, and $\X$ be a  functorially finite rigid subcategory of $\C$ such that $\proj(\C)\subseteq\X$. If $(\cal U,\cal V)$ is an $\X$-mutation pair in $\C$, then
\begin{enumerate}
\item[\emph{(1)}] $\cal U$ is a rigid subcategory of $\C$ if and only if so is $\cal V$.
\item[\emph{(2)}] $\cal U$ is a {\red cluster tilting} subcategory of $\C$ if and only if so is $\cal V$.
\end{enumerate}
\end{theorem}

\proof By Proposition \ref{z6}, we have that $\cal N:=\X^{\perp_{\E}}/\X$ is a $2$-Calabi-Yau triangulated category.
By Lemma \ref{z4}, we obtain that $(\overline{\cal U},\overline{\cal V}):=(\cal U/\X,\cal V/\X )$ forms $0$-mutation pair in $\cal N$.
Thus we have $\overline{\cal V}=\overline{\cal U}\langle1\rangle$. In particular, $\overline{\cal U}$ is a rigid subcategory
(resp. cluster tilting subcategory) of $\cal N$ if and only if $\overline{\cal V}$ is a rigid subcategory
(resp. cluster tilting subcategory) of $\cal N$. On the other hand, by Theorem \ref{z8}, we have that
$\overline{\cal U}$ (resp. $\overline{\cal V}$) is a rigid subcategory
(resp. cluster tilting subcategory) of $\cal N$ if and only if $\cal U$ (resp. $\cal V$) is a rigid subcategory
(resp. cluster tilting subcategory) of $\C$. Thus the assertions follow.  \qed

\begin{definition}\label{defnz11}
Let $\C$ be a $2$-Calabi-Yau extriangulated category.
We call a functorially finite rigid subcategory $\X$ of $\C$ such that $\proj(\C)\subseteq\X$ \emph{almost complete cluster tilting}
if there exists a cluster tilting subcategory $\cal R$ of $\C$ such that $\X\subseteq\cal R$ and $\cal R=\X\cup{\red \add R_0}$,
where $R_0$ is an indecomposable object which is not isomorphic to any object in $\X$.
Such $R_0$ is called a \emph{complement} of a almost complete cluster tilting subcategory $\X$.
\end{definition}

\begin{theorem}\label{thmz11}
Let $\C$ be a $2$-Calabi-Yau extriangulated category.
Then any almost complete cluster tilting subcategory $\X$ of $\C$ is contained in exactly two
cluster tilting subcategories $\cal R$ and $\cal Q$ of $\C$.
Both $(\cal R,\cal Q)$ and $(\cal Q,\cal R)$ form $\X$-mutation pairs.
\end{theorem}

\proof We have that $\cal N:=\X^{\perp_{\E}}/\X$ is a $2$-Calabi-Yau triangulated category, and  $0$ is a almost complete cluster tilting subcategory of $\cal N$.
Since any almost complete cluster tilting subcategory are exactly two cluster tilting subcategory in $2$-Calabi-Yau triangulated category, see \cite[Theorem 5.3]{iy}. Then the object $0$ in $\cal N$ has two complements, say $Q_0$, $R_0$, and both ($Q_0, R_0)$ $(R_0,Q_0)$ form {\red $0$-matation pairs} in $\cal N$. By Theorem \ref{z8}, we have that $\X$ {\red are} contained in exactly two
cluster tilting subcategories: $\cal R=\X\cup\add R_0$, $\cal Q=\X\cup\add Q_0$. It is easy to see that  $(\cal R,\cal Q)$ and $(\cal Q,\cal R)$ form $\X$-mutation pairs. \qed
\vspace{3mm}

For almost complete cluster tilting object, we have the following.

\begin{corollary}\label{z12}
Let $\C$ be a $2$-Calabi-Yau extriangulated category. Then any basic almost complete cluster tilting object $R$ of $\C$ such that $\proj(\C)\subseteq\add R$ is a direct summand of exactly two basic
cluster tilting objects in $\C$.
\end{corollary}

Now for an almost complete cluster tilting subcategory $\X$ of $\C$, assume that $Q_0,R_0$ are two complements of $\X$ in Theorem \ref{thmz11}. Then there are two $\E$-triangles related to
$Q_0, R_0$:
$$
\xymatrix{Q_0\ar[r]^{a}&X\ar[r]^{b}&R_0\ar@{-->}[r]^{\delta}&}$$
$$
\xymatrix{R_0\ar[r]^{a'}&X'\ar[r]^{b'}&Q_0\ar@{-->}[r]^{\delta'}&}
$$
where $a, a'$ are the minimal left $\X$-approximations, and $b,b'$ are the minimal right $\X$-approximations.
These two $\E$-triangles are called \emph{exchange $\E\textrm{-}$triangles}.

\section{Cotorsion pairs in a $2$-Calabi-Yau extriangulated category}
 Let $\C$ be a $2$-Calabi-Yau extriangulated category with a cluster tilting object. We will give a classification of cotorsion pairs in $\C$ in the first subsection and study the cluster
 structures in a cotorsion pair inherited from $\C$ in the second subsection.

\subsection{Classification of cotorsion pairs}
Let $\C$ be a Frobenius extriangulated category. For any objects $A,B\in\C$, by Lemma \ref{e5}, we have a functorially isomorphism $$\E(A,B)\simeq\Hom_{\overline{\C}}(A,B\langle 1\rangle).$$
Thus $\C$ is $2$-Calabi-Yau if and only if the stable category
$\overline{\C}$ is $2$-Calabi-Yau.

\begin{lemma}\label{one-to-one}
Let $\C$ be a $2$-Calabi-Yau extriangulated category. Then
\begin{itemize}
\item[\emph{(1)}] $\X$ is a rigid subcategory of $\C$ if and only if $\overline{\X}$ is  a rigid subcategory of $\overline{\C}$.

\item[\emph{(2)}] $\X$ is a cluster tilting subcategory of $\C$ if and only if $\overline{\X}$ is  a cluster tilting subcategory of $\overline{\C}$.
    \end{itemize}
\end{lemma}

\proof {\red This is a specific case of Theorem \ref{z8}.} \qed

\begin{proposition}\label{prop:decomposition of category}
 Let $\C$ be a $2$-Calabi-Yau extriangulated category with a cluster tilting object. Let $(\X,\Y)$ be a cotorsion pair in $\C$. Then  the core $\cal I=\add I$, for some rigid object which containing all indecomposable projective objects in $\C$
 and there exists a decomposition of triangulated category $^{\perp_\E}\I/\I=\X/\I\oplus \Y/\I$.
\end{proposition}

\proof {\red Since $\C$ is a $2$-Calabi-Yau extriangulated category with a cluster tilting object, the $\mbox{Proj}(\C)=\mbox{Inj} (\C) =\add P$ for a projective-injective object $P$ in $\C$, and the stable category $\overline{\C}$ of $\C$ is a $2$-Calabi-Yau triangulated category with a cluster tilting object. Therefore, as  a rigid subcategory in $\overline{\C}$, $\I =\add I_1$, where $I_1$ is a rigid object in $\overline{\C}$.}  Then the core $\cal I=\add I$, where $I=I_1\oplus P$ is a rigid object $I$ in $\C$. It is easy to see that $\X={^{\perp_{\E}}\Y}$ and $\Y=\X^{\perp_{\E}}$.
It is straightforward to check that $(\X,\X)$ forms a $\I$-mutation pair. Since $\X$ is extension closed, we have that
$\X/\I$ is a triangulated category.

By Proposition \ref{z6}, we have that $^{\perp_\E}\I/\I$ is a $2$-Calabi-Yau triangulated category.
By Lemma \ref{one-to-one}, $^{\perp_\E}\I/\I$ has a cluster tilting object.
Since $(\X,\Y)$ is a cotorsion pair with core $\I$,
by Corollary \ref{corollary:correspondence between cotorsion pairs}, we have that
$(\overline{\X},\overline{\Y})$ is a cotorsion pair in the stable category $\overline{\C}$ with core $\overline{\I}$.
By Corollary 4.5 in \cite{zz2}, we obtain that $(\overline{\Y},\overline{\X})$ is a cotorsion pair in $\overline{\C}$ with core $\overline{\I}$.
By Corollary \ref{corollary:correspondence between cotorsion pairs}, we have that
$(\Y,\X)$ is a cotorsion pair in $\C$ with core $\I$. Thus we obtain that
$\X$ is functorially finite subcategory in $\C$. By Lemma II.2.2 in \cite{birs}, we have a decomposition
$$^{\perp_\E}\I/\I=\X/\I\oplus \Y/\I.$$  \qed

Combining this result with the decomposition theorem of $2$-Calabi-Yau triangulated categories with a cluster tilting object in \cite{zz2}, we have the classification of cotorsion pairs with
given core $\I$ in a $2$-Calabi-Yau extriangulated category with a cluster tilting object.

\begin{theorem}
  Let $\C$ be a $2$-Calabi-Yau extriangulated category with a cluster tilting object and $\I$ a
rigid subcategory of $\C$ such that $\proj(\C)\subseteq\I$. Let $^{\perp_\E}\I/\I=\oplus_{j\in J}A_j$ be the complete decomposition of
$^{\perp_\E}\I/\I$ (where all $A_j$ are indecomposable triangulated categories). Then
\begin{itemize}
\item[\emph{1}.] all cotorsion pairs with core $\I$ are obtained
as preimages under $\pi:{}^{\perp_\E}\I\rightarrow {}^{\perp_\E}\I/\I$ of the pairs
$(\oplus_{j\in L}A_j, \oplus_{j\in J-L}A_j)$ where $L$ is a subset of $J$. There are
$2^{ns({}^{\perp_\E}\I/\I)}$ cotorsion pairs with core $\I$, where  $ns(\C)$ is the number of indecomposable direct summands of such decomposition of $\C$.
\item[\emph{2}.] $(\X,\Y)$ {\red is} a cotorsion pair with core $\I$ if and only if so is $(\Y,\X)$.\end{itemize}
\end{theorem}

\proof The first assertion follows from Proposition \ref{prop:decomposition of category} and \cite[Theorem 4.4]{zz2}, the second is a consequence of the first one.  \qed
\vspace{2mm}

As an application to exact stably $2$-Calabi-Yau categories with cluster tilting objects, we get a classification of cotorsion pairs in these categories.

\begin{corollary}
 Let $\cal B$ be an exact stably $2$-Calabi-Yau category with a cluster tilting object and $\I$ a
rigid subcategory of $\cal B$ such that $\proj(\C)\subseteq\I$. Let $^{\perp_1}\I/\I=\oplus_{j\in J}A_j$ be the complete decomposition of
the triangulated category $^{\perp_1}\I/\I$. Then all cotorsion pairs with core $\I$ are obtained
as preimages under $\pi:{}^{\perp_1}\I\rightarrow {}^{\perp_1}\I/\I$ of the pairs
$(\oplus_{j\in L}A_j, \oplus_{j\in J-L}A_j)$ where $L$ is a subset of $J$. There are
$2^{ns({}^{\perp_1}I/I)}$ cotorsion pairs with core $\I$, where  $ns(\cal B)$ is the number of indecomposable direct summands of such decomposition of $\cal B$ and $\I^{\perp_{1}}=\{M\in\cal B\mid \emph{\Ext}^{1}_{\cal B}(\I,M)=0\}$.
 Moreover if $(\X,\Y)$ is a cotorsion pair in $\cal B$, then so is $(\Y,\X)$.
\end{corollary}

From this corollary, we can get a classification of cotorsion pairs in the module categories over preprojective algebras of Dynkin quivers or in the categories $\mathcal{C}_M$, where $M$ is a terminal module over preprojective algebras; these categories are used to categorify some cyclic cluster algebras with coefficients by Gei{\ss}-Leclerc-Schr\"{o}er in \cite{gls2,gls3}. This classification on cotorsion pairs may be used to study the cluster subalgebras of the cluster algebras categorified by these categories.

\subsection{Cluster structures in cotorsion pairs}

Any extension closed subcategory of an extriangulated category is an extriangulated category, we can talk about cluster tilting subcategories in it.
\begin{definition}
Let $\X$ be a contravariantly finite (or covariantly finite) extension closed
subcategory of an extriangulated category $\C$.
{\red A functorially finite subcategory $\cal R$ of $\X$ is called an $\X$-cluster tilting subcategory provided that for an object $R$ of $\X$, $R\in\cal R$ if and only if $\E(R,X)=0$ for any object $X\in\X$ if and only if $\E(X,R)=0$ for any object $X\in\X$.}
An object $R$ in $\X$ is called an $\X$-cluster tilting object if $\add R$ is an $\X$-cluster tilting subcategory.
\end{definition}

The relation on cluster tilting subcategory between $\C$ and its cotosion pair $(\X,\Y)$ is given by the following:

 \begin{proposition}\label{decomp of cto}
Let $\C$ be a $2$-Calabi-Yau extriangulated category category with a cluster tilting object, and  $(\X,\Y)$ be a cotorsion pair in $\C$ with core $\I$. Then
\begin{itemize}
\item[\emph{1}.] Any cluster tilting subcategory $\T$ containing $\I$ can be written uniquely as: $\T=\T_{\X}\oplus \I\oplus  \T_{\Y}$,
such that $ \T_{\X}\oplus \I$ is $\X$-cluster tilting, and $ \T_{\Y}\oplus \I$ is $\Y$-cluster tilting.

\item[\emph{2}.]  Any $\X$-cluster tilting subcategory (or $\Y$-cluster tilting subcategory) contains $\I$, and can be written as $\T_{\X}\oplus \I$ ( $\T_{\Y}\oplus \I$ resp.).
Furthermore $ \T_{\X}\oplus \I\oplus  \T_{\Y}$ is a cluster tilting subcategory in $\C$.

\item[\emph{3}.]  There is a bijection between the set of cluster tilting subcategories containing $\I$ in $\C$ and the product of the set of $\X$-cluster tilting subcategories with the set of $\Y$-cluster tilting
  subcategories. The bijection is given by $ \T\mapsto (\T_{\X}\oplus \I,  \T_{\Y}\oplus \I), $   where $\T=\T_{\X}\oplus \I\oplus \T_{\Y}$.
\end{itemize}
\end{proposition}

\proof The proof of Proposition 5.5 in \cite{zz2} works also in this setting with the help of
Theorem \ref{z1}.   \qed

Recall that a quiver is a quadruple $(Q_0,Q_1,s,t)$ consisting of two sets: $Q_0$ (the set of vertices) and $Q_1$ (the set of arrows), and of two maps $s, t$ which map each arrow $\alpha \in Q_1$ to its source $s(\alpha)$ and its target $t(\alpha)$, respectively. An ice quiver is a quiver $Q$ associated a subset $F$ (the set of frozen vertices) of its vertex set. The full subquiver of $Q$ with vertex set $Q_0\setminus F$ (the set of exchangeable vertices) is call the exchangeable part of $Q$.

\begin{definition}\label{def£ºquiver of category}
 Let $\C$ be a $2$-Calabi-Yau extriangulated category  with cluster tilting objects.
\begin{enumerate}
\item For a cluster tilting subcategory $\T$ in a $2$-Calabi-Yau extriangulated category $\C$, we define $Q(\T)$ as an ice quiver whose exchangeable vertices are the isomorphism classes of indecomposable objects in $\T$ which are not projective objects and the frozen vertices are the isomorphism classes of indecomposable projective objects. For two vertices  $T_i$ and $T_j$ (not both the frozen vertices), the number of arrows from $T_i$ to $T_j$ is the dimension of irreducible morphism space $irr(T_i,T_j)=rad(T_i,T_j)/rad^2{(T_i,T_j)}$ in $\T$.
\item For {\red an} $\X$-cluster tilting subcategory $\T_{\X}\oplus \I$ in $\X$, we define $Q(\T_{\X}\oplus \I)$ as an ice quiver whose exchangeable vertices are the isomorphism classes of indecomposable objects in $\T_\X$ and the frozen vertices are the isomorphism classes of indecomposable objects in $\I$. For two vertices  $T_i$ and $T_j$ (not both the frozen vertices), the number of arrows from $T_i$ to $T_j$ is the dimension of irreducible morphism space $irr(T_i,T_j)$ in $\T_{\X}\oplus \I)$. The quiver $Q(\T_{\Y}\oplus \I)$ of $\Y$-cluster tilting subcategory is defined similarly.
\item For a cluster tilting subcatgory $\overline{\T}$ in $\overline{\C}$, we define $Q(\overline{\T})$ as a quiver whose (exchangeable) vertices are the isomorphism classes of indecomposable objects of $\overline{\T}$. For two vertices  $\overline{T_i}$ and $\overline{T_j}$, the number of arrows from $\overline{T_i}$ to $\overline{T_j}$ is the dimension of irreducible morphism space $irr(\overline{T_i},\overline{T_j})$ in $\overline{\T}$.
\end{enumerate}
\end{definition}

The cluster structure in a $2$-Calabi-Yau triangulated category or in an exact stably $2-$Calabi-Yau category is defined in \cite{birs} and \cite{fk}. This structure is given by cluster tilting subcategories and `categorifies' the cluster algebra associated to the quivers of the cluster tilting subcategories. Then one can use a cluster map \cite{birs} to transform a cluster structure in these categories to the cluster algebra. The cluster map is also called the cluster character in \cite{p}. In our setting, we similarly define a cluster structure in $2$-Calabi-Yau extriangualted categories with cluster tilting subcategories and assume that there always {\red is} a cluster map from the cluster structure to a cluster algebra. We omit the related definitions and the detailed discussions, and refer to \cite{birs} and \cite{fk} for more details. However, we have the following equivalent description of {\red a} cluster structure, where the cases of the triangulated category and the exact category are proved in Theorem II 1.6 of \cite{birs}.

\begin{theorem}\label{thm of cluster structure}
Let $\C$ be a $2$-Calabi-Yau extriangulated category with cluster tilting objects.
If $\C$ has no loops or $2$-cycles (this means the quivers of any cluster tilting objects contains neither loops nor $2$-cycles), then the cluster tilting subcategories determine a
cluster structure for $\C$.
\end{theorem}
\proof Similar to the proof of Theorem II.1.6 in \cite{birs}.\qed
\vspace{2mm}

From now {\red on} to the end of the paper, we fix the following settings. Let $\C$ be a $2$-Calabi-Yau extriangulated category with cluster tilting objects, and cluster tiling objects form a cluster structure. Let $(\X,\Y)$ be a cotorsion pair of $\C$ with core $\I$. Then by Proposition \ref{decomp of cto}, we can write a cluster tilting object $\T$ as $\T_\X\oplus \I \oplus \T_\Y$ with $\T_\X\oplus \I$ being $\X$-cluster tilting and {\red $\T_\Y\oplus \I$} being $\Y$-cluster tilting.

 \begin{proposition}
Under above settings,
\begin{enumerate}
\item the quiver $Q(\overline{\T})$ is obtained from $Q(\T)$ by deleting all the frozen vertices and arrows connected to these vertices.
\item the quiver $Q(\T_{\X}\oplus \I)$ is obtained from $Q(\T)$ by deleting all the isomorphism classes of indecomposable direct summands of $\T_{\Y}$ and arrows connected to these vertices, and freezing the isomorphism classes of indecomposable direct summands of $\I$.
\end{enumerate}
\end{proposition}
\begin{proof}
\begin{enumerate}
\item Since an object becomes zero object in $\overline{\C}$ if and only if it belongs to the projective subcategory $Proj$, the vertices of $Q(\overline{\T})$ are just the exchangeable vertices of $Q(\T)$. It can be  directly derived from the homomorphism theorem of groups that, for any two exchangeable vertices $T_1$ and $T_2$ in $Q(\T)$, the dimension of $irr(T_1,T_2)$ in $\C$ is equal to the dimension of $irr(\overline{T_1},\overline{T_2})$ in $\overline{\C}$. Then by the  the Definition \ref{def£ºquiver of category} of the quivers, we are done.
\item We only need to show that for any two vertices $T_1$ and $T_2$ of $Q(\T_{\X}\oplus \I)$ (not both frozen vertices), the dimension of  $irr(T_1,T_2)$ in $\T$ is the same as the dimension of $irr(T_1,T_2)$ in $\T_\X\oplus \I$, or equivalently, a map $f$ from $T_1$ to $T_2$ is irreducible in $\T$ if and only if it is irreducible in $\T_\X\oplus \I$. For the convenience, we assume that $T_1$ is an indecomposable direct summand of $\T_\X$. It is clear that if $f$ is irreducible in $\T$, then it is irreducible in $\T_\X\oplus \T$. Conversely, for an irreducible map $f$ in $\T_\X\oplus \I$, if it is not irreducible in $\T$, then it factor through an object $T_3$ in $\T_\Y$. We write $f$ as a composition of $f_1\colon T_1\rightarrow T_3$ and $f_2\colon T_3\rightarrow T_2$. By Proposition \ref{prop:decomposition of category}, $f_1$ factors through an object $I_1$ in $\I$, and we write it as a composition $g_1\colon T_1\rightarrow I_1$ and $g_2\colon I_1 \rightarrow T_3$. Note that $g_1$ and $f_2g_2$ are both in the radical space of $\X\oplus \I$, this is contradict to the assumption that $f$ is irreducible in {\red $\T_\X\oplus \I$}. Therefore $f$ is irreducible in $\T$.
\end{enumerate}
\end{proof}

\begin{lemma}
Under above settings, there are cluster structures in $\overline{\C}$, $\X$ and $\Y$, which are induced from the cluster structure of $\C$. We call the cluster structures in $\X$ and $\Y$ the cluster {\red substructures} of the cluster structure in $\C$.
\end{lemma}
\begin{proof}
It follows from Proposition \ref{decomp of cto} and the above proposition that the categories $\overline{\C}$, $\X$ and $\Y$ have no loops or $2$-cycles. It follows from Theorem \ref{thm of cluster structure} that all of these categories, as extriangulated categories, have cluster structures, which are induced from the cluster structure of $\C$.
\end{proof}

\section{Relation between cotorsion pairs and pairs of rooted cluster subalgebras}

In this section, we study cotorsion pairs in a $2$-Calabi-Yau extriangulated category with cluster tilting object in the point of view categorification. In the first subsection, we recall some basics on the cluster algebras introduced in \cite{fz} and the rooted cluster algebras introduced in \cite{ads}. In the second subsection, we give a correspondence between cotorsion pairs with certain pairs of rooted cluster subalgebras.
\subsection{Rooted cluster subalgebras}
 We assume in this section that there are no loops or 2-cycles, and no arrows between frozen vertices in the ice quivers. Let $m+n=|Q_0|$ the number of vertices in an ice quiver $Q$, and denote vertices by $Q_0=\{1, 2, \cdot \cdot \cdot, m+n\}$ and frozen vertices by $F=\{m+1, m+2, \cdot \cdot \cdot, m+n\}$. By associating each vertex $1\leqslant i \leqslant m+n$ an indeterminate element $x_i$, we have a set $\x=\{x_1, x_2, \cdot \cdot \cdot, x_{m+n}\}$. Denote by  $\ex=\{x_1, x_2, \cdot \cdot \cdot, x_{m}\}$ and $\fx=\{x_{m+1}, x_{m+2}, \cdot \cdot \cdot, x_{m+n}\}$. The cluster algebra $\A_Q$ is a $\Z-$subalgebra of the rational function field $\F=\Q(x_1,\cdot \cdot \cdot, x_{m+n})$ generated by generators obtained recursively from $\x$ in the following manner. For an exchangeable vertex $1\leq i \leq m$, we obtain a new triple $(\mu_i(\ex), \fx, \mu_i(Q))$ from $(\ex, \fx, Q)$ by a mutation $\mu_i$ at $i$ defined as follows.

Firstly, $\mu_i(Q)=(Q')$ is obtained by:

{\red
\begin{itemize}
\item[(a)] inserting a new arrow $\gamma: j\to k$ for each path $j\stackrel{\alpha}\rightarrow i \stackrel{\beta}\rightarrow k$;
\item[(b)] reversing all arrows incident with $i$;
\item[(c)] removing a maximal collection of pairwise disjoint 2-cycles and removing arrows between frozen vertices.
\end{itemize}
}

Secondly, $\mu_i(\ex)=(\ex\setminus\{x_i\}) \cup \{x_i^\prime\}$ where $x_i^\prime\in{\F}$ is defined by the following exchange relation:
\begin{equation*}
x_i x_i^\prime = \prod_{\alpha:i\to j} x_j + \prod_{\alpha:j\to i} x_j.
\end{equation*}
Denote by ${\X}$ the union of all possible sets of variables obtained from $\x$ by successive mutations. Then the cluster algebra ${\A_Q}$ is the $\mathbb{Z}$-subalgebra of ${\F}$ generated by ${\X}$. We call each triple $\tilde{\S}=(\tilde{\ex}, \tilde{\fx}, \tilde{Q},)$ obtained from $\S=(\ex, \fx, Q)$ by successive mutations a seed, and $\tilde{\x}=\tilde{\ex}\sqcup\tilde{\fx}$ a cluster.
The elements $\tilde{x}_1,\dots, \tilde{x}_{m+n}$ of a cluster $\tilde{\x}$ are cluster variables. The variables in the subset $\tilde{\ex}$ are exchangeable cluster variables and the variables in the subset $\tilde{\fx}$ are frozen cluster variables.
A rooted cluster algebra associated to $\A_Q$ is a pair $(\S,\A_Q)$ (or $\A_\S$ for brevity).

For the aim of constructing a category frame work of studying cluster algebras, rooted cluster morphisms are introduced in \cite{ads}. They are special ring homomorphisms between rooted cluster algebras which are compatible with cluster mutations.
A sequence $(x_1, \cdots , x_l)$ is called $(\ex, \fx, Q)$-admissible if $x_1\in \ex$ and $x_i$ is in $\mu_{x_{i-1}} \circ \cdots \circ \mu_{x_1}(\ex)$  for every $2 \leqslant i \leqslant l$. A rooted cluster morphism $f$ from $\A_\S$ to $\A_{\S'}$ is a ring homomorphism satisfies the following three conditions:
\begin{itemize}
\item[(a)] $f(\ex) \subset \ex' \sqcup \Z$~;
\item[(b)] $f(\fx) \subset \x' \sqcup \Z$~;
\item[(c)] For every $(f,\x,\x')$-biadmissible sequence $(x_1, \cdots , x_l)$, we have $f(\mu_{x_l} \circ \cdots \circ \mu_{x_1,\x}(y)) = \mu_{f(x_l)} \circ \cdots \circ \mu_{f(x_1),\x'}(f(y))$ for any $y$ in $\x$. Here a $(f,\x,\x')$-biadmissible sequence $(x_1, \cdots , x_l)$ is a $\S$ -admissible sequence such that $(f(x_1),\cdots ,f(x_l))$ is $\S'$ -admissible.
\end{itemize}

\begin{definition}
We call $\A_\S$ a rooted cluster subalgebra of $\A_{\S'}$ if there is an injective rooted cluster morphism from $\A_\S$ to $\A_{\S'}$.
\end{definition}
We collect some definitions and results from \cite{cz} to introduce the complete pair of a rooted cluster algbera.

\begin{defn-prop}\label{defn-prop:cluster subalgebras}
Let $Q$ be an ice quiver.
\begin{itemize}
\item[(a)] We call $Q$ indecomposable if it is connected and its exchangeable part is also connected. We call a seed indecomposable if its quiver is indecomposable.
\item[(b)] We call a full subquiver $Q'$ of $Q$ a connected component if the exchangeable part is connected and the frozen vertices are all the frozen vertices in $Q$ which are connected directly to the exchangeable vertices of $Q'$.
\item[(c)] Let $\S=(\ex, \fx, Q)$ be a seed and $\ex'$ be a subset of $\ex$. We call $\S_f=(\ex\setminus \ex', \fx\sqcup \ex', Q_f)$ the freezing of $\S$ at $\ex'$, where $Q_f$ is obtained from $Q$ by freezing vertices corresponding to elements in $\ex'$ and deleting the arrows between these vertices.
\item[(d)] Let $\S_1=(\ex_1, \fx_1, Q_1)$ and $\S_2=(\ex_2, \fx_2, Q_2)$ be two seeds. Assume that there is a bijection between $\fx_1$ and $\fx_2$, then we obtain a quiver $Q$ by gluing $Q_1$ and $Q_2$ (as graphs) together at frozen vertices unified under this bijection. A gluing of $\S_1$ and $\S_2$ is a seed $\S=(\ex, \fx, Q)$, where $\ex=\ex_1 \sqcup \ex_2$ and $\fx=\fx_1$ (or equivalently $\fx_2$).
\item[(e)] Let $\S$ be a seed and $\S_f$ be a freezing of $\S$. Then $\A_{\S_f}$ is a rooted cluster subalgebra of $\A_\S$. Let $\S'$ be a gluing of some connected components of $\A_{\S_f}$, then $\A_{\S'}$ is a rooted cluster subalgebra of $\A_{\S_f}$, and thus a rooted cluster subalgebra of $\A_\S$. Conversely, all the rooted cluster subalgebras of $\A_\S$ is obtained from the above two ways. For more details, see \cite{cz}.
\end{itemize}
\end{defn-prop}

\begin{definition}\label{def of rooted cluster algebra}
Let $\S=(\ex, \fx, Q)$ be a seed and $\S_f$ be the freezing of $\S$ at a subset $\ex'$ of $\ex$. Denote by $\fx_0$ the set of isolated frozen variables of $\S_f$. Let $\S_1=(\ex_1, \fx_1, Q_1)$ and $\S_2=(\ex_2, \fx_2, Q_2)$ be two seeds. We call the pair $(\A(\S_1), \A(\S_2))$ a complete pair of subalgebras of $\A(\S)$ with coefficient set $\fx \sqcup \ex'$ if the following conditions are satisfied:
\begin{enumerate}
\item both $\S_1$ and $\S_2$ are gluings of some indecomposable components of $\S_f$;
\item $\ex_1\cap \ex_2 = \emptyset$ and $\ex = \ex_1 \sqcup \ex_2 \sqcup \ex'$;
\item $\fx\cup\fx_0\subseteq \fx_1\cap\fx_2$.
\end{enumerate}
\end{definition}

%We recall a Frobenius version of Proposition II.2.3 in \cite{BIRS09}, which plays a key role in the following study of cluster substructures in cotorsion pairs.

%\begin{lemma}
%Let $\C$ be a stably 2-CY category. For any cotorsion pair $(\X,\Y)$ of $\C$,
%\begin{itemize}
%\item[1.] the core $I=\X\cap \Y$ is a functorially finite rigid subcategory;
%\item[2.] both $\X$ and $\Y$ are stably 2-CY categories;
%\item[3.] all of $\X/I$, $\Y/I$ and $^{\bot}I[1]/I$ are 2-CY triangulated categories.
%\end{itemize}
%\end{lemma}

%\begin{lemma}
%Let $\C$ be a stably 2-CY category. Then there is a decomposition of triangulated category $^{\bot}I[1]/I=\X/I\oplus \Y/I$.
%\end{lemma}

\subsection{Relations with cluster algebras}

Let $\C$ be a $2$-Calabi-Yau extriangulated category with (non-zero) cluster tilting objects. Fix a cluster tilting object $T$,
{\red we denote by $\T$ the full subcategory of $\C$ additive generated by $T$.} Let $(\X,\Y)$ be a cotorsion pair with core $\I$.
We denote the rooted cluster algebras corresponding to the quivers $Q(\T)$, $Q(\overline{\T})$, $Q(\T_{\X}\oplus \I)$ and $Q(\T_{\Y}\oplus \I)$ as $\A(\T)$, $\A(\overline{\T})$, $\A(\T_\X\oplus \I)$ and $\A(\T_\Y\oplus \I)$ respectively.
%\begin{itemize}
%\item $R(\cal T):$ the subcategory of $\C$ additive generated by rigid objects which are reachable from  $\T$ by a finite number of mutations, where the mutation is defined in Theorem \ref{thmz11}.
%\item $R(\cal T_\X\oplus \I):$ the subcategory of $\C$ additive generated by rigid objects which are reachable from $\cal T_\X\oplus \I$ by finite number of mutations.
%\item $R(\cal T_\Y\oplus \I):$ the subcategory of $\C$ additive generated by rigid objects which are reachable from $\cal T_\Y\oplus \I$ by finite number of mutations.
%\end{itemize}

\begin{theorem}\label{thm of Last}
Let $\C$ be a  $2$-Calabi-Yau extriangulated category with cluster tilting objects, which form a cluster structure.{\red Assume that there is a cluster map $\chi$ from $\C$ to $\A(\T)$. Then $\chi$ induces a correspondence $(\X,\Y)\mapsto (\A(\T_{\X}\oplus \I),\A(\T_{\Y}\oplus \I))$, which gives a bijection:}
\begin{center}
\{cotorsion pairs in $\C$ with core $\I$\}\\
$\Updownarrow$\\
\{complete pairs of rooted cluster subalgebras of $\A(\T)$ with coefficient set $\chi(\I)$\}.
\end{center}
This bijection induces the following bijection:
\begin{center}
\{t-structures in $\C$\}\\
$\Updownarrow$\\
\{complete pairs of rooted cluster subalgebras of $\A(\T)$ with coefficients set $\chi(\pj)$\}.
\end{center}
\end{theorem}

\begin{proof}
For a cotorsion pair $(\X,\Y)$, we have a decomposition of categories $^{\perp_\E}\I/\I=\X/\I\oplus \Y/\I$ from Proposition \ref{prop:decomposition of category}. Therefore morphisms between any two objects representing vertices $T_1 \in \T_\X$ and $T_2 \in \T_\Y$ factor through objects in $\I$. So there are no arrows between $T_1$ and $T_2$ in $Q(\T)$, and thus $Q(\T_{\X}\oplus \I)$ and $Q(\T_{\Y}\oplus \I)$ are both gluings of some indecomposable components of $Q(\T)$. The rest two conditions of complete pairs in Definition \ref{def of rooted cluster algebra} are clear. Therefore $(\A(\T_{\X}\oplus\I),\A(\T_{\Y}\oplus \I))$ is a complete pair of rooted cluster subalgebras of $\A(\T)$ with coefficient set $\chi(I)$. Conversely, if we have a complete pair $(\A',\A'')$ of rooted cluster subalgebras of $\A(\T)$.
From Proposition-Definition \ref{defn-prop:cluster subalgebras} (e) and the definition of complete pairs of rooted cluster subalgebras, the rooted cluster subalgebras $\A'$ and $\A''$ are of the forms $\A(\T'\oplus \I)$ and $\A(\T''\oplus \I)$ with $\T=\T'\oplus \I \oplus\T''$, and in the quiver $Q(\T)$, there are no arrows between vertices in $\T'$ and vertices in $\T''$. Now we consider the pair $(\T',\T'')$ in the subfactor category $^{\perp_\E}\I/\I$, which is a $2$-Calabi-Yau triangulated category by Proposition \ref{z6}. Note that $\T'\oplus\T''$ is a cluster tilting subcategory in $^{\perp_\E}\I/\I$ by Proposition \ref{decomp of cto} and we have $\Hom_{^{\perp_\E}\I/\I}(\T',\T'')=\Hom_{^{\perp_\E}\I/\I}(\T'',\T')=0$. Thus we have $^{\perp_\E}\I/\I=(\T'\oplus\T'')\ast (\T'\langle1\rangle\oplus\T''\langle1\rangle)=\T'\ast \T'\langle1\rangle \oplus \T''\ast \T''\langle1\rangle=\C_1\oplus \C_2$ as a decomposition of triangulated category by Proposition 3.5 in \cite{zz2}.  Let $\pi\colon ^{\perp_\E}I\rightarrow {^{\perp_\E}\I}/\I$ be the natural projection. Then because $(\C_1,\C_2)$ is a cotorsion pair in $^{\perp_\E}\I/\I$ with core $\{0\}$, $(\X',\Y')=(\pi^{-1}(\C_1),\pi^{-1}(\C_2))$ is a cotorsion pair in $\C$ with core $\I$ by Corollary \ref{corollary:correspondence between cotorsion pairs}. It is clear that the above two kinds of processes are the inverse processes with each other. Since the $t$-structures are those cotorsion pairs with cores $\pj$, the second correspondence is clear from the first one.
\end{proof}

\begin{remark} Applying the theorem to  $2$-Calabi-Yau triangulated categories with cluster tilting objects, we recover the corresponding results in \cite{cz}. The application to exact stably $2$-Calabi-Yau categories $\C$ with cluster tilting objects yields a new correspondence between cotorsion pairs in the categories $\C$ and the complete pairs of cluster subalgebras of the cluster algebras categorified by $\C$.  We state it here in the following:

\end{remark}

\begin{corollary}\label{car of Last}
Let $\C$ be an exact stably $2$-Calabi-Yau category with cluster tilting objects,
 which form a cluster structure, $\I$ a
rigid subcategory of $\C$ such that $\proj(\C)\subseteq\I$.
{\red Assume that there is a cluster map $\chi$ from $\C$ to $\A(\T)$. Then $\chi$ induces a correspondence $(\X,\Y)\mapsto (\A(\T_{\X}\oplus \I),\A(\T_{\Y}\oplus \I))$, which gives a bijection:}
\begin{center}
\{cotorsion pairs in $\C$ with core $\I$\}\\
$\Updownarrow$\\
\{complete pairs of rooted cluster subalgebras of $\A(\T)$ with coefficient set $\chi(\I)$\}.
\end{center}
\end{corollary}

\section{An example}
Here we explain the correspondence in Theorem \ref{thm of Last} by an example relating to a Grassmannian cluster algebra. Recall that Fomin and Zelevinsky \cite{fz} proved that the homogenous coordinate ring of the Grassmannian $G_{2,n}$ of 2-planes in n-space has a cluster algebra structure. Scott \cite{s} generalized this result to each Grassmannian $G_{r,n}$ by using Postnikov arrangements. Since such a cluster algebra has non-trivial coefficients, the Frobenius category is the candidate to categorify it \cite{birs, fk}. In 2008, Gei{\ss}-Leclerc-Shr\"{o}er \cite{gls1} described such a category in the representation category of a preprojective algebra, which provided a {\red categorification} the cluster algebra structure on the coordinate ring of the big cell in the Grassmannian. Recently, by relating to this result, Jensen-King-Su \cite{jks} found a ring $A_{r,n}$ whose category of maximal Cohen-Macaulay modules is a Frobenius category, which categorified the coordinate ring of Grassmannian $G_{r,n}$. Specially, we consider the Frobenius category relating to $E_6$ type cluster algebra of $G_{3,7}$.

\begin{example}

Recall that under the Plucker embedding $G_{3,7}\rightarrow \mathbb{P}(\wedge^3\mathbb{C}^7)$, the coordinate ring $\mathbb{C}(G_{3,7})$ of $G_{3,7}$ is generated by the Plucker coordinates $x_A$ of $3$-subset $A$ of  \{1,2,3,4,5,6,7\} which obey the Plucker relations. By viewing $\mathbb{C}(G_{3,7})$ as a cluster algebra, these coordinates are special cluster variables and these relations are special cluster exchange relations \cite{s}. Following \cite{s}, we consider the initial quiver of $\mathbb{C}(G_{3,7})$ depicted in Figure \ref{fig:initial quiver of CG37}, where we denote the vertex of cluster variable $x_A$ as $\Delta^A$ and frame the frozen vertices.
\begin{figure}
$$\xymatrix{&\framebox{$\Delta^{712}$}\ar[d]&\framebox{$\Delta^{123}$}\ar[dr]&\framebox{$\Delta^{234}$}&\\
&\Delta^{137}\ar[ur]\ar[d]&&\Delta^{134}\ar[ll]\ar[u]&\\
&\Delta^{147}\ar[rr]\ar[dl]&&\Delta^{347}\ar[u]\ar[d]&\\
\framebox{$\Delta^{671}$}\ar[r]&\Delta^{467}\ar[u]\ar[d]&&\Delta^{457}\ar[ll]\ar[r]&\framebox{$\Delta^{345}$}\ar[ul]\\
&\framebox{$\Delta^{567}$}&&\framebox{$\Delta^{456}$}\ar[u]&}
$$
\caption{The initial quiver of the cluster algebra $\mathbb{C}(G_{3,7})$}
\label{fig:initial quiver of CG37}
\end{figure}

After ordered mutations at points $\Delta^{147},\Delta^{347},\Delta^{137},\Delta^{134},\Delta^{467},\Delta^{134}$ and $\Delta^{457}$, we obtain a new quiver with cluster $$\x=\ex\sqcup\fx=\{x_{126},x_{135},x_{235},x_{367},x',x''\}
\sqcup\{x_{123},x_{234},x_{345},x_{456},x_{567},x_{671},x_{712}\},$$ which consists of ten Plucker coordinates and other two variables  $x'$ and $x''$. The new quiver $Q$ is in Figure \ref{fig:The standard quiver of E6 type}, which is the standard bipartite quiver of $E_6$ type after deleting the frozen vertices. We still denote the point of Plucker coordinate $x_A$ as $\Delta^A$, and $\Delta'$ and $\Delta''$ are the points of the two exceptional cluster variables $x'$ and $x''$ respectively.

\begin{figure}
$$\xymatrix{\framebox{$\Delta^{712}$}&\framebox{$\Delta^{671}$}\ar[d]\ar[r]&
\Delta^{367}\ar[d]&\framebox{$\Delta^{567}$}\ar[d]&\framebox{$\Delta^{234}$}\ar[d]\\
\Delta^{126}\ar[u]\ar[dr]&\Delta'\ar[r]\ar[l]&\Delta''\ar[ul]\ar[ur]\ar[dr]\ar[d]&\Delta^{135}\ar[l]\ar[r]&
\Delta^{235}\ar[dl]\ar[dll]\\
&\framebox{$\Delta^{456}$}\ar[u]&\framebox{$\Delta^{123}$}\ar[ul]\ar[ur]&
\framebox{$\Delta^{345}$}\ar[u]\ar[uul]&
}$$
\caption{The standard quiver $Q$ of the cluster algebra $\mathbb{C}(G_{3,7})$}
\label{fig:The standard quiver of E6 type}
\end{figure}

Then a new seed $\S=(\ex,\fx,Q)$ of the cluster algebra $\mathbb{C}(G_{3,7})$ is obtained. Now we fix the exchangable cluster variable $x''$ in $\ex$, while $\Delta''$ becomes frozen in $Q$. Then note that $Q$ can be separated as three indecomposable quivers at vertex $\Delta''$, correspondingly, the rooted cluster algebra $(\mathbb{C}(G_{3,7}),\S)$ have three indecomposable rooted cluster subalgebra $\A',\A''$ and $\A'''$ with coefficients set $\{x''\}\sqcup\fx$. The quivers of these rooted cluster algebras are the $Q',Q''$ and $Q'''$ in Figure \ref{fig:subquiver 1}, \ref{fig:subquiver 2} and \ref{fig:subquiver 3} respectively.

\begin{figure}
$$\xymatrix{\framebox{$\Delta^{234}$}&\framebox{$\Delta^{712}$}&\framebox{$\Delta^{671}$}\ar[d]&&\\
\framebox{$\Delta^{345}$}&\Delta^{126}\ar[u]\ar[dr]&\Delta'\ar[r]\ar[l]&\framebox{$\Delta''$}&\\
\framebox{$\Delta^{567}$}&&\framebox{$\Delta^{456}$}\ar[u]&\framebox{$\Delta^{123}$}\ar[ul]&
}$$
\caption{The quiver $Q'$ of rooted cluster subalgebra $\A'$ of $\mathbb{C}(G_{3,7})$}
\label{fig:subquiver 1}
\end{figure}

\begin{figure}
$$\xymatrix{\framebox{$\Delta^{456}$}&&
\framebox{$\Delta^{567}$}\ar[d]&\framebox{$\Delta^{234}$}\ar[d]\\
\framebox{$\Delta^{671}$}&\framebox{$\Delta''$}&\Delta^{135}\ar[l]\ar[r]&
\Delta^{235}\ar[dl]\ar[dll]\\
\framebox{$\Delta^{712}$}&\framebox{$\Delta^{123}$}\ar[ur]&
\framebox{$\Delta^{345}$}\ar[u]&
}$$
\caption{The quiver $Q''$ of rooted cluster subalgebra $\A''$ of $\mathbb{C}(G_{3,7})$}
\label{fig:subquiver 2}
\end{figure}

\begin{figure}
$$\xymatrix{&\framebox{$\Delta^{671}$}\ar[r]&
\Delta^{367}\ar[d]&\framebox{$\Delta^{345}$}\ar[l]&\\
&&\framebox{$\Delta''$}&&\\
\framebox{$\Delta^{123}$}&\framebox{$\Delta^{234}$}&\framebox{$\Delta^{456}$}&
\framebox{$\Delta^{567}$}&\framebox{$\Delta^{712}$}
}$$
\caption{The quiver $Q'''$ of rooted cluster subalgebra $\A'''$ of $\mathbb{C}(G_{3,7})$}
\label{fig:subquiver 3}
\end{figure}
So there are $C_3^0+C_3^1+C_3^2+C_3^3=8$ complete pairs of rooted cluster subalgebra of $(\mathbb{C}(G_{3,7}),\S)$ with coefficient set $\{x''\}\sqcup\fx$.

Now we recall the Frobenius category $\C$ from \cite{jks} and correspond these complete subalgebra pairs to the cotorsion pairs in the category with the core corresponding to cluster variables in $\{x''\}\sqcup\fx$. In \cite{jks}, Jensen-King-Su described a complete algebra $A_{3,7}$ as a quotient algebra of preprojective algebra of type $\tilde{A_7}$ up to some relations. The category $\mathsf{CM}(A_{3,7})$ of maximal Cohen-Macaulay $A_{3,7}$-modules is a Hom-finite Krull-Schmidt stably 2-Calabi-Yau Frobenius category with cluster tilting objects. The Auslander-Reiten quiver of $\mathsf{CM}(A_{3,7})$ is depicted in Figure \ref{fig:AR quiver of CMA}, where the objects are represented by `profiles' of the modules, see Section 6 of \cite{jks} for details.

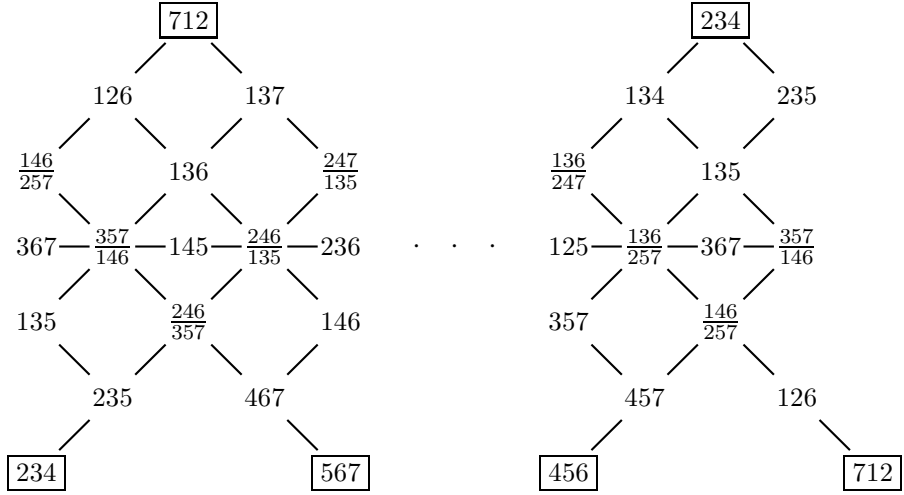
\begin{figure}
\begin{center}
{\begin{tikzpicture}

\node[] (C) at (0,1)
						{$\frac{146}{257}$};
\node[] (C) at (0,0)
						{\small{367}};
\node[] (C) at (0,-1)
						{\small{135}};
\node[] (C) at (0,-3)
						{\small{\framebox{$234$}}};

\node[] (C) at (1,2)
						{\small{126}};
\node[] (C) at (1,0)
						{$\frac{357}{146}$};
\node[] (C) at (1,-2)
						{\small{235}};

\node[] (C) at (2,3)
						{\small{\framebox{$712$}}};
\node[] (C) at (2,1)
						{\small{136}};
\node[] (C) at (2,0)
						{\small{145}};
\node[] (C) at (2,-1)
						{$\frac{246}{357}$};

\node[] (C) at (3,2)
						{\small{137}};
\node[] (C) at (3,0)
						{$\frac{246}{135}$};
\node[] (C) at (3,-2)
						{\small{467}};

\node[] (C) at (4,1)
						{$\frac{247}{135}$};
\node[] (C) at (4,0)
						{\small{236}};
\node[] (C) at (4,-1)
						{\small{146}};
\node[] (C) at (4,-3)
						{\small{\framebox{$567$}}};

\node[] (C) at (5,0)
						{$\cdot$};
\node[] (C) at (5.5,0)
						{$\cdot$};
\node[] (C) at (6,0)
						{$\cdot$};

\draw[thick] (0.3,1.3) -- (0.7,1.7);
\draw[thick] (0.3,0.7) -- (0.7,0.3);
\draw[thick] (0.3,0) -- (0.7,0);
\draw[thick] (0.3,-0.7) -- (0.7,-0.3);
\draw[thick] (0.3,-1.3) -- (0.7,-1.7);
\draw[thick] (0.3,-2.7) -- (0.7,-2.3);

\draw[thick] (1.3,2.3) -- (1.7,2.7);
\draw[thick] (1.3,1.7) -- (1.7,1.3);
\draw[thick] (1.3,0.3) -- (1.7,0.7);
\draw[thick] (1.3,0) -- (1.7,0);
\draw[thick] (1.3,-0.3) -- (1.7,-0.7);
\draw[thick] (1.3,-1.7) -- (1.7,-1.3);

\draw[thick] (2.3,2.7) -- (2.7,2.3);
\draw[thick] (2.3,1.3) -- (2.7,1.7);
\draw[thick] (2.3,0.7) -- (2.7,0.3);
\draw[thick] (2.3,0) -- (2.7,0);
\draw[thick] (2.3,-0.7) -- (2.7,-0.3);
\draw[thick] (2.3,-1.3) -- (2.7,-1.7);

\draw[thick] (3.3,1.7) -- (3.7,1.3);
\draw[thick] (3.3,0.3) -- (3.7,0.7);
\draw[thick] (3.3,0) -- (3.7,0);
\draw[thick] (3.3,-0.3) -- (3.7,-0.7);
\draw[thick] (3.3,-1.7) -- (3.7,-1.3);
\draw[thick] (3.3,-2.3) -- (3.7,-2.7);
%%%%%%%%%%%%%%%%%%%%%%%%%%%%%%%%%%%%%%%%%%%%%%%%%%%%%%%%%%%%
\node[] (C) at (7,1)
						{$\frac{136}{247}$};
\node[] (C) at (7,0)
						{\small{125}};
\node[] (C) at (7,-1)
						{\small{357}};
\node[] (C) at (7,-3)
						{\small{\framebox{$456$}}};

\node[] (C) at (8,2)
						{\small{134}};
\node[] (C) at (8,0)
						{$\frac{136}{257}$};
\node[] (C) at (8,-2)
						{\small{457}};

\node[] (C) at (9,3)
						{\small{\framebox{$234$}}};
\node[] (C) at (9,1)
						{\small{135}};
\node[] (C) at (9,0)
						{\small{367}};
\node[] (C) at (9,-1)
						{$\frac{146}{257}$};

\node[] (C) at (10,2)
						{\small{235}};
\node[] (C) at (10,0)
						{$\frac{357}{146}$};
\node[] (C) at (10,-2)
						{\small{126}};

\draw[thick] (7.3,1.3) -- (7.7,1.7);
\draw[thick] (7.3,0.7) -- (7.7,0.3);
\draw[thick] (7.3,0) -- (7.7,0);
\draw[thick] (7.3,-0.7) -- (7.7,-0.3);
\draw[thick] (7.3,-1.3) -- (7.7,-1.7);
\draw[thick] (7.3,-2.7) -- (7.7,-2.3);

\draw[thick] (8.3,2.3) -- (8.7,2.7);
\draw[thick] (8.3,1.7) -- (8.7,1.3);
\draw[thick] (8.3,0.3) -- (8.7,0.7);
\draw[thick] (8.3,0) -- (8.7,0);
\draw[thick] (8.3,-0.3) -- (8.7,-0.7);
\draw[thick] (8.3,-1.7) -- (8.7,-1.3);

\draw[thick] (9.3,2.7) -- (9.7,2.3);
\draw[thick] (9.3,1.3) -- (9.7,1.7);
\draw[thick] (9.3,0.7) -- (9.7,0.3);
\draw[thick] (9.3,0) -- (9.7,0);
\draw[thick] (9.3,-0.7) -- (9.7,-0.3);
\draw[thick] (9.3,-1.3) -- (9.7,-1.7);

\node[] (C) at (11,-3)
						{\small{\framebox{$712$}}};
\draw[thick] (10.7,-2.7) -- (10.3,-2.3);
\end{tikzpicture}}
\end{center}
\begin{center}
\caption{The Auslander-Reiten quiver of $\mathsf{CM}(A_{3,7})$}
\label{fig:AR quiver of CMA}
\end{center}
\end{figure}
The quiver is periodical with two sides coincide under a Mobius reflection. The shape of omitted part is a two times repetition of the left side quiver. Here the framed objects in $T_{pi}=123\oplus234\oplus345\oplus456\oplus567\oplus671\oplus712$ are projective-injective objects and $T=126\oplus\frac{146}{257}\oplus\frac{357}{146}\oplus135\oplus235\oplus367\oplus123\oplus234\oplus345\oplus456\oplus567\oplus671\oplus712$
is a cluster tilting object. Under the cluster map
$\chi:\mathsf{CM}(A_{3,7})\rightarrow \mathbb{C}(G_{3,7})
${\red \cite{jks}(see also \cite{fk,birs})}
,$T$ is mapped to $\x$, where the explicit correspondence is given in Table \ref{table:cluster map correspondence}.
\begin{table}[ht]
\begin{equation*}
\begin{array}{cc}
\textrm{Rigid object in $T$} & \textrm{Cluster variable in $\x$} \\
\hline
ijk & x_{ijk}\\[1.5mm]
\frac{146}{257} & x'\\[1.5mm]
\frac{357}{146} & x''
\end{array}
\end{equation*}
\smallskip
\caption{{\red Cluster map correspondence on the cluster tilting object $T$}}
\label{table:cluster map correspondence}
\end{table}
With this correspondence, the quiver of $\End_{\mathsf{CM}(A_{3,7})}(T)$ is isomorphic to the quiver of the cluster $\x$, that is the $Q$ in Figure
\ref{fig:The standard quiver of E6 type}. Moreover, the separation of $Q$ at $\Delta''$ corresponds to a separation of $T$ as three indecomposable objects $T_1=126\oplus\frac{146}{257}\oplus\frac{357}{146}\oplus123\oplus234\oplus345\oplus456
\oplus567\oplus671\oplus712$,
$T_2=\frac{357}{146}\oplus135\oplus235\oplus123
\oplus234\oplus345\oplus456\oplus567\oplus671\oplus712$ and
$T_3=\frac{357}{146}\oplus367
\oplus123\oplus234\oplus345\oplus456\oplus567\oplus671\oplus712$, whose quivers are isomorphic to $Q_1$, $Q_2$ and $Q_3$ respectively. There are three indecomposable functionally finite extension-closed subcategories $\mathcal{X}_1$, $\mathcal{X}_2$ and $\mathcal{X}_3$ in $\mathsf{CM}(A_{3,7})$ with projective-injective object $I=\frac{357}{146}\oplus T_{pi}$. We draw up the Auslander-Reiten quivers of them in Figure \ref{fig:AR quiver of cotortion pairs 1}, \ref{fig:AR quiver of cotortion pairs 2} and \ref{fig:AR quiver of cotortion pairs 3} respectively. The objects $T_1$, $T_2$ and $T_3$ are cluster tilting objects of $\mathcal{X}_1$, $\mathcal{X}_2$ and $\mathcal{X}_3$ respectively. These subcategories make up eight cotorsion pairs of $\mathsf{CM}(A_{3,7})$ with core $I=\frac{357}{146}\oplus T_{pi}$. Finally, they correspond to eight complete subalgebras pairs of $(\mathbb{C}(G_{3,7}),\S)$ consisted of rooted cluster subalgebras $\A'$, $\A''$ and $\A'''$ described above.

\begin{figure}
\begin{center}
{\begin{tikzpicture}

\node[] (C) at (-1,0)
						{\small{\framebox{$671$}}};
\node[] (C) at (-1,1)
						{\small{\framebox{$567$}}};
\node[] (C) at (-1,2)
						{\small{\framebox{$345$}}};
\node[] (C) at (-1,3)
						{\small{\framebox{$234$}}};
\node[] (C) at (0,1)
						{$\frac{146}{257}$};

\node[] (C) at (1,2)
						{\small{126}};
\node[] (C) at (1,0)
						{\small{\framebox{$\frac{357}{146}$}}};

\node[] (C) at (2,3)
						{\small{\framebox{$712$}}};

\node[] (C) at (4,3)
						{\small{\framebox{$671$}}};

\node[] (C) at (2,1)
						{\small{136}};

\node[] (C) at (3,2)
						{\small{137}};

\node[] (C) at (4,1)
						{$\frac{247}{135}$};

\node[] (C) at (5,0)
						{\small{\framebox{$123$}}};

\node[] (C) at (5,2)
                        {\small{$\frac{146}{257}$}};
\node[] (C) at (7,0)
						{\small{\framebox{$712$}}};
\node[] (C) at (6,1)
						{$\small{126}$};

\draw[thick] (-0.7,0.3) -- (-0.3,0.7);
\draw[thick] (0.3,1.3) -- (0.7,1.7);
\draw[thick] (4.3,1.3) -- (4.7,1.7);
\draw[thick] (3.3,2.3) -- (3.7,2.7);
\draw[thick] (3.3,0.3) -- (3.7,0.7);
\draw[thick] (5.3,0.3) -- (5.7,0.7);

\draw[thick] (1.3,2.3) -- (1.7,2.7);
\draw[thick] (1.3,1.7) -- (1.7,1.3);

\draw[thick] (2.3,2.7) -- (2.7,2.3);
\draw[thick] (2.3,1.3) -- (2.7,1.7);

\draw[thick] (0.3,0.7) -- (0.7,0.3);

\draw[thick] (1.3,1.7) -- (1.7,1.3);
\draw[thick] (1.3,0.3) -- (1.7,0.7);

\draw[thick] (3.3,1.7) -- (3.7,1.3);

%%%%%%%%%%%%%%%%%%%%%%%%%%%%%%%%%%%%%%%%%%%%%%%%%%%%%%%%%%%%

\node[] (C) at (3,0)
						{\small{\framebox{$456$}}};

\draw[thick] (5.3,1.7) -- (5.7,1.3);
\draw[thick] (4.3,2.7) -- (4.7,2.3);
\draw[thick] (6.3,0.7) -- (6.7,0.3);
\draw[thick] (4.3,0.7) -- (4.7,0.3);

\draw[thick] (2.3,0.7) -- (2.7,0.3);

\end{tikzpicture}}
\end{center}
\begin{center}
\caption{The Auslander-Reiten quiver of $\chi_1$}
\label{fig:AR quiver of cotortion pairs 1}
\end{center}
\end{figure}
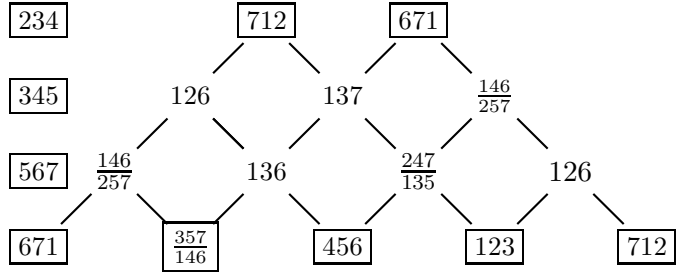

\begin{figure}
\begin{center}
{\begin{tikzpicture}

\node[] (C) at (-1,0)
						{\small{\framebox{$712$}}};
\node[] (C) at (-1,1)
						{\small{\framebox{$671$}}};
\node[] (C) at (-1,2)
						{\small{\framebox{$456$}}};

\node[] (C) at (1,2)
						{\small{135}};
\node[] (C) at (1,0)
						{\small{\framebox{$234$}}};
\node[] (C) at (3,0)
						{\small{345}};
\node[] (C) at (2,3)
						{\small{\framebox{$\frac{357}{146}$}}};

\node[] (C) at (4,3)
						{\small{\framebox{$123$}}};
\node[] (C) at (6,3)
						{\small{\framebox{$234$}}};

\node[] (C) at (2,1)
						{\small{235}};

\node[] (C) at (3,2)
						  {\small{$\frac{246}{357}$}};

\node[] (C) at (4,1)
						{$\small{467}$};

\node[] (C) at (5,0)
						{\small{\framebox{$567$}}};

\node[] (C) at (5,2)
                        {$\small{146}$};
\node[] (C) at (7,2)
                        {$\small{235}$};
\node[] (C) at (7,0)
						{\small{\framebox{$\frac{357}{146}$}}};
\node[] (C) at (6,1)
						{$\small{135}$};

\draw[thick] (4.3,1.3) -- (4.7,1.7);
\draw[thick] (3.3,2.3) -- (3.7,2.7);
\draw[thick] (3.3,0.3) -- (3.7,0.7);
\draw[thick] (5.3,0.3) -- (5.7,0.7);

\draw[thick] (1.3,2.3) -- (1.7,2.7);
\draw[thick] (1.3,1.7) -- (1.7,1.3);

\draw[thick] (2.3,2.7) -- (2.7,2.3);
\draw[thick] (2.3,1.3) -- (2.7,1.7);

\draw[thick] (1.3,1.7) -- (1.7,1.3);
\draw[thick] (1.3,0.3) -- (1.7,0.7);

\draw[thick] (3.3,1.7) -- (3.7,1.3);

\draw[thick] (5.3,2.3) -- (5.7,2.7);

\draw[thick] (6.3,1.3) -- (6.7,1.7);

\draw[thick] (6.3,2.7) -- (6.7,2.3);

%%%%%%%%%%%%%%%%%%%%%%%%%%%%%%%%%%%%%%%%%%%%%%%%%%%%%%%%%%%%

\node[] (C) at (3,0)
						{\small{\framebox{$345$}}};

\draw[thick] (5.3,1.7) -- (5.7,1.3);
\draw[thick] (4.3,2.7) -- (4.7,2.3);
\draw[thick] (6.3,0.7) -- (6.7,0.3);
\draw[thick] (4.3,0.7) -- (4.7,0.3);

\draw[thick] (2.3,0.7) -- (2.7,0.3);

\end{tikzpicture}}
\end{center}
\begin{center}
\caption{The Auslander-Reiten quiver of $\chi_2$}
\label{fig:AR quiver of cotortion pairs 2}
\end{center}
\end{figure}

\begin{figure}
\begin{center}
{\begin{tikzpicture}

\node[] (C) at (1,-1)
						{\small{\framebox{$123$}}};
\node[] (C) at (2,-1)
						{\small{\framebox{$234$}}};
\node[] (C) at (3,-1)
						{\small{\framebox{$456$}}};
\node[] (C) at (4,-1)
						{\small{\framebox{$567$}}};
\node[] (C) at (5,-1)
						{\small{\framebox{$712$}}};

\node[] (C) at (1,2)
						{\small{\framebox{$345$}}};
\node[] (C) at (1,0)
						{\small{\framebox{$671$}}};

\node[] (C) at (2,1)
						{\small{367}};
\node[] (C) at (3,1)
						{\small{\framebox{$\frac{357}{146}$}}};
\node[] (C) at (4,1)
						{\small{145}};

\node[] (C) at (5,0)
						{\small{\framebox{$345$}}};
\node[] (C) at (5,2)
						{\small{\framebox{$671$}}};

\draw[thick] (1.3,1.7) -- (1.7,1.3);
\draw[thick] (1.3,0.3) -- (1.7,0.7);

\draw[thick] (2.3,1) -- (2.6,1);
\draw[thick] (3.4,1) -- (3.7,1);

\draw[thick] (5.3,1.7) -- (5.7,1.3);
\draw[thick] (5.3,0.3) -- (5.7,0.7);

\draw[thick] (4.3,1.3) -- (4.7,1.7);
\draw[thick] (4.3,0.7) -- (4.7,0.3);
\node[] (C) at (6,1)
						{\small{367}};

\end{tikzpicture}}
\end{center}
\begin{center}
\caption{The Auslander-Reiten quiver of $\chi_3$}
\label{fig:AR quiver of cotortion pairs 3}
\end{center}
\end{figure}
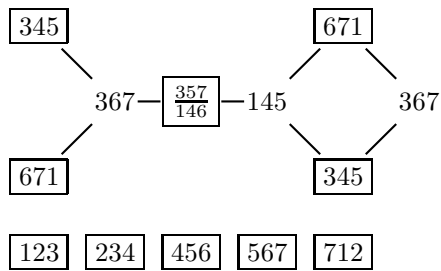
\end{example}

\section*{Acknowlegement} The correspondence in Corollary 5.6 was asked by Osamu Iyama in a conference at Seoul Korea in 2014. We are grateful to him for this.
The authors also would like to thank the referees for reading the paper carefully and for many suggestions on mathematics and English expressions. The third author thanks Tiwei Zhao for his comments on the paper!

Wen Chang\\
School of Mathematics and Information Science, Shaanxi Normal University, Xi'an 710062, P. R. China.\\
E-mail: \verb"changwen161@163.com"\\[0.3cm]
Panyue Zhou\\
College of Mathematics, Hunan Institute of Science and Technology,
414006, Yueyang, Hunan,  P. R. China.\\
E-mail: \verb"panyuezhou@163.com"\\[0.3cm]
Bin Zhu\\
Department of Mathematical Sciences, Tsinghua University,
100084 Beijing, P. R. China.\\
E-mail: \verb"bzhu@math.tsinghua.edu.cn"

\end{document}